%% file: main.tex
\documentclass[11pt]{amsart}

\usepackage{graphicx} 
\usepackage{amssymb,amsmath,amsthm,amsrefs}

\usepackage{amsfonts}
\usepackage{amstext}
\usepackage{xcolor}
\usepackage[toc,page]{appendix}

\usepackage[utf8]{inputenc}
\usepackage[english]{babel}
\usepackage{mathptmx} 
\DeclareMathAlphabet{\mathcal}{OMS}{cmsy}{m}{n}
\usepackage{pgfplots}
\usepackage{float}

\usepackage{mathrsfs}
\usepackage[english]{babel}
\usepackage{tikz}
\usepackage{blindtext}
\usepackage{enumerate}
\usepackage{hyperref}
\usepackage{verbatim}
\usepackage{graphicx} 

\newtheorem{theorem}[equation]{Theorem}
\newtheorem*{theorem*}{Theorem}
\newtheorem{lemma}[equation]{Lemma}

\newtheorem{proposition}[equation]{Proposition}

\newtheorem{corollary}[equation]{Corollary}

\newtheorem{definition}[equation]{Definition}

\newtheorem{remark}[equation]{Remark}
\newtheorem{notation}[equation]{Notation}

\newtheorem{assumption}[equation]{Assumption}
\newtheorem{conjecture}[equation]{Conjecture}
\numberwithin{equation}{section}

\newcommand{\mres}{\mathbin{\vrule height 1.6ex depth 0pt width 0.13ex\vrule height 0.13ex depth 0pt width 1.3ex}}

\newcommand{\R}{\mathbb{R}}

\newcommand{\CCC}{\mathsf{C}}
\newcommand{\UUU}{\mathsf{U}}

\newcommand{\Ccal}{\mathcal{C}}
\newcommand{\Acal}{\mathcal{A}}
\newcommand{\Grp}{\mathscr{G}}
\newcommand{\abs}[1]{\left\lvert#1\right\rvert}

\newcommand{\supp}{\mathrm{supp}}
\newcommand{\uC}{{\mathscr{C}}}
\newcommand{\uS}{{\mathscr{S}}}

\title{Self-expanders of Positive Genus}
\author[G.~Shao]{Guanhua~Shao}
\address{Department of Mathematics, Rutgers University, Pistacaway, NJ 08854} 
\email{gs977@math.rutgers.edu}
	   
\author[J.~Zou]{Jiahua~Zou} 
\address{Department of Mathematics, Rutgers University, Pistacaway, NJ 08854} 
\email{jiahua.zou@rutgers.edu} 

\begin{document}
\date{\today}


\begin{abstract}
For a general class of cones in $\mathbb{R}^3$, we construct self-expanders of positive genus asymptotic to these cones. As a result, we use these self-expanders to construct a mean curvature flow with genus strictly decreasing but not to zero at the first singular time. We also construct a sequence of self-expanders with unbounded genus which are asymptotic to the same rotationally symmetric cone.
\end{abstract}

\maketitle
\section{Introduction}
\subsection*{Background and the main result}
A hypersurface $\Sigma^n \subset \mathbb{R}^{n+1}$ is called a \emph{self-expander} if 
\begin{equation*}
    H_{\Sigma} = \frac{1}{2} X_{\Sigma}\cdot \nu_{\Sigma}.
\end{equation*}
Here for the mean curvature we use the convention $H_{\Sigma}:=\Delta_{\Sigma} X_{\Sigma}\cdot \nu_{\Sigma}$, where $X_{\Sigma}$ is the embedding of $\Sigma$, $\Delta_{\Sigma}$ is the Laplacian operator on $\Sigma$ induced from the Euclidean metric in $\R^{n+1}$, and $\nu_{\Sigma}$ is a choice of unit normal vector. Self-expanders arise naturally in the study of mean curvature flow. Indeed, $\Sigma$ is a self-expander if and only if the family of homothetic hypersurfaces
\begin{equation*}
	(0,\infty)\ni t\mapsto \sqrt{t}\Sigma 
\end{equation*}
is a mean curvature flow (MCF). 

Self-expanders model the behavior of a MCF as it emerges from a conical singularity \cite{angenent1995computed} as well as the long time behavior of the flow \cite{Ecker1989Mean}. On the other hand, from the point of view of minimal surfaces, self-expanders are critical points of the functional
\begin{equation}\label{eq:E}
E[\Sigma]:= \int_{\Sigma} e^{\frac{\abs{X}^2}{4}} \, d\mathcal{H}^n,
\end{equation}
i.e., minimal surfaces in the conformal metric
\begin{equation}\label{eq:gexp}
 g_{Exp}:=e^{\frac{\abs{X}^2}{2n}}g_{Euc}.
\end{equation}

Self-expanders have been studied extensively by various authors. Ecker and Huisken \cite{Ecker1989Mean} studied the mean curvature flow of entire graphs. They used normalized flows to show that for any smooth graphical cone $\Ccal$, there is a unique self-expander asymptotic to $\Ccal$ (cf. \cite[p.2]{Bernstein2022Topological}) (Lemma \ref{graphical self-expander} below). Using ODE analysis and numerical evidence, Angenent-Ilmanen-Chopp \cite{angenent1995computed} gave explicit examples of rotationally symmetric self-expanders and fattening phenomena for the level set flows starting from them. In \cite[p. 13-14]{ilmanen1998lectures}, Ilmanen showed the existence of $E$-minimizing self-expanders asymptotic to prescribed cones in the Euclidean space using variational constructions. Ding \cite{ding2020minimal} used barriers to give an alternative proof of this existence result by Ilmanen \cite[p.13-14]{ilmanen1998lectures}. And he showed that for any sufficiently smooth mean-convex but not minimizing cone, the corresponding $E$-minimizing self-expander is strictly mean convex. Bernstein and Wang developed a degree-theoretic method to produce asymptotically conical self-expanders of prescribed topological type in \cite{Bernstein2021The,Bernstein2021Smooth, Bernstein2023Integer} and a min-max theory for asymptotically conical self-expanders in \cite{Bernstein2022Mountain}.

In spite of their expected prevalence, to the authors' knowledge, there is no explicit example of a self-expander with positive genus in $\mathbb{R}^3$. In this paper, for a general class of cones, we show that there always exist connected self-expanders of positive genus asymptotic to these cones. 
\begin{definition}\label{def:Pi}
     For any $g \in \mathbb{N}$ and $\delta > 0$, we define $\Pi[g, \delta]$ to be a collection of smoothly embedded cones in $\mathbb{R}^3$ where each $\Ccal[g] \in \Pi[g, \delta]$ satisfies the following.
    \begin{enumerate}[(i)]
        \item $\Ccal[g] \subset \{x^2 + y^2 > \delta z^2\}$.
        \item The link $\mathcal{L}[\Ccal[g]]$ of $\Ccal[g]$ has $3$ components $\mathcal{L}_1, \mathcal{L}_2, \mathcal{L}_3$.
        \item The unique simple cone $\Ccal_i[g]$ with link $\mathcal{L}_i$ is graphical for each $i \in \{1, 2, 3\}$.
        \item  $\Ccal[g]$ has $\Grp[g]$-symmetry.
        \end{enumerate}
\end{definition}
The group $\Grp[g] \cong D_{2(g + 1)}$ 
is defined in \eqref{eq:G} below. Our main result is the following.

\begin{theorem} \label{the main result of the paper}
    There is a universal constant $\delta_0 > 0$ such that for any $g \in \mathbb{N}$ and $\Ccal[g] \in \Pi[g, \delta_0]$, there exists a complete connected $\Grp[g]$-equivariant self-expander $\Sigma[g]=\Sigma[\Ccal[g]]$ of genus $g$. Moreover, $\Sigma[g]$ is asymptotic to the cone $\Ccal[g]$. 
\end{theorem}

\begin{remark}[Different genera coming from the same cone and fattening] \label{fattening of the level set flow of the cone}
For any cone $\Ccal[g] \in \Pi[g, \delta_0]$, By Lemma \ref{graphical self-expander} below, for each $i \in \{1, 2, 3\}$, there is a unique self-expander $\Lambda_i$ asymptotic to $\Ccal_i$. Moreover, each $\Lambda_i$ is graphical. Let $\Lambda : = \bigcup_{i =1}^3 \Lambda_i$ be the self-expander consisting of these three disjoint graphical sheets. Then $\Lambda$ and $\Sigma[g]$ are both asymptotic to $\Ccal[g]$. However, $\Lambda$ has genus zero, while $\Sigma[g]$ has positive genus $g$.
By \cite[Theorem 2.4]{Chodosh2024Conical}, this non-uniqueness of self-expanders coming from the same cone will then lead to fattening of the level set flow starting from $\Ccal[g]$.
\end{remark}

\begin{remark} \label{remark of the main result}
Let $\breve{\Ccal}[g]$ be the asymptotic cone of the shrinker $\breve{M}[g+1,1]$ we constructed in \cite{GZ1} with genus $2g$ and three ends. By \cite[Theorem 8.1]{GZ1} and , for all $g \in \mathbb{N}$ sufficiently large, clearly $\breve{\Ccal}[g] \in \Pi[g, \delta_0]$ holds. Thus, by Theorem \ref{the main result of the paper} above, we can construct a complete connected genus-$g$ self-expander $\breve{\Sigma}[g]$ asymptotic to $\breve{\Ccal}[g]$ for all $g \in \mathbb{N}$ sufficiently large.
\end{remark}
\subsection*{Strict genus reduction}
In his lecture notes \cite{ilmanen1998lectures}, Ilmanen raised the following \emph{Genus Reduction Conjecture}:
 \begin{conjecture} [{\cite[p.23]{ilmanen1998lectures}}] \label{Genus Reduction Conjecture}
     Let $M_t$ be any mean curvature flow in $\mathbb{R}^3$. Then the genus of $M_t$ strictly decreases at the time of a singularity, unless the singularity is a neck-pinch or a shrinking sphere.
 \end{conjecture}

White \cite[Theorem 1]{White1995Topology} (cf. \cite[p. 20]{ilmanen1998lectures}) first showed that along the level set flow, genus of the compact outermost flow is non-increasing. His result \cite[Theorem 1]{White1995Topology} implies that the MCF shall simplify the topology of the surfaces. Recently, by the works of Brendle \cite{Brendle2016Embedded}, Ilmanen-White \cite[p.21]{ilmanen1998lectures}, Chodosh-Schulze \cite[Corollary 1.2]{Chodosh2021Uniqueness}, Chodosh-Daniels-Schulze \cite{Chodosh2024Conical}, Bamler-Kleiner \cite{bamler2023multiplicity}, Chodosh-Choi-Schulze-Mantoulidis \cite{Chodosh2024Generic,Chodosh2023Mean}, Conjecture \ref{Genus Reduction Conjecture} has been fully resolved at non-generic singularities for outermost flows. Combining Remark \ref{remark of the main result} with the shrinker $\breve{M}_g=\breve{M}[g+1,1]$ we constructed in \cite[Theorem 8.1]{GZ1}, we have the first example of an MCF whose genus strictly decreases at a singular time but does not drop to zero.
\begin{corollary} \label{Genus Drop example of a MCF}
For each $g \in \mathbb{N}$ sufficiently large, there is a mean curvature flow $\{M_t^g\}_{t \in \mathbb{R}}$ in $\mathbb{R}^3$ satisfying the following:
\begin{enumerate}[(i)]
  \item For $t < 0$, $M_t^g := \sqrt{-t} \cdot \breve{M}_g$, where $M_t^g$ is smooth and self-similarly shrinking. Moreover, $M_t^g$ has genus $2g$.
    \item At $t = 0$, $M_0^g = \breve{\Ccal}[g]$ is a triple cone with a unique singularity at $x = 0$.
    \item For $t > 0$, $M_t^g : = \sqrt{t} \cdot \breve{\Sigma}[g]$ is smooth and self-similarly expanding. Moreover, $M_t^g$ has genus $g$.
\end{enumerate}
\end{corollary}
At the singular time $t = 0$, the genus of the MCF $\{M_t^g\}$ drops from $2g$ to $g$, which is consistent with Conjecture \ref{Genus Reduction Conjecture} that a singularity does consume a certain amount of topology.

\begin{remark}[Closed embedded surfaces that fatten under MCF] \label{closed embedded surfaces fatten under MCF}
Observe that by \cite{lee2024closed}, for each $g \in \mathbb{N}$, there is a closed embedded surface $M_g$ such that the MCF starting from $M_g$ will develop a singularity modeled by the shrinker $\breve{M}_g$. By \cite[Corollary 1.2]{Chodosh2021Uniqueness} and \cite[Theorem 1.1]{Chodosh2024Conical}, at the first singular time, the MCF starting from $M_g$ is modeled on the level set flow of $\breve{\Ccal}_g$. Consequently, by Remark \ref{fattening of the level set flow of the cone}, the MCF starting from $M_g$ must fatten at its first singular time. Note that the existence of such closed smoothly embedded surfaces that fatten under MCF were pioneered by \cite{ilmanen2024fattening,ketover2024self,HMW}.
\end{remark}
Recently, Hoffman-Martin-White \cite{HMW} used mean curvature flow methods to generate self-shrinkers that has the $\Grp[g]$-symmetry, genus $2g$, and three ends, for all $g \in \mathbb{N}$ sufficiently large. We expect that their shrinkers are the same as the shrinkers $\breve{M}_g$ constructed by gluing PDE methods \cite{GZ1} for all $g \in \mathbb{N}$ sufficiently large. Therefore, we expect that the same genus drop in Corollary \ref{Genus Drop example of a MCF} and the fattening phenomena described in Remark \ref{closed embedded surfaces fatten under MCF} also apply to the shrinkers constructed by them \cite{HMW}.
\subsection*{Infinitely many self-expanders of unbounded genus coming from the same cone} Another immediate application of our main result is that we can construct a sequence of self-expanders with unbounded genera asymptotic to the same cone. 

Indeed, define
\begin{equation}\label{eq:Cepsilon}
    \Ccal_{\epsilon}:=\Ccal(\epsilon)\cup \{z=0\}.
\end{equation}
where $\Ccal(\epsilon) : = \{ x^2 + y^2 = \epsilon z^2 \}$ is the double cone with slope $1/\epsilon$. As it is obvious that for $ \epsilon > \delta_0$, $\Ccal_{\epsilon} \in \Pi[g, \delta_{0}]$ for all $g \in \mathbb{N}$, we then have the following corollary by Theorem \ref{the main result of the paper}.
\begin{corollary} [Proposition \ref{INfinitely many self-expanders}] \label{Infintely many expanders coming from the same cone}
For any $ \epsilon > \delta_0$, there is a sequence of self-expanders $\{ \Sigma[g]=\Sigma[g,\epsilon]\}_{g \in \mathbb{N}}$ such that the following hold.
\begin{enumerate}[(i)]
    \item $\Sigma[g]$ is asymptotic to $\Ccal_{\epsilon}$.
    \item $\Sigma[g]$ is connected and $g(\Sigma[g]) = g$.
    \item $\Sigma[g]$ has $\Grp[g]$-symmetry.
\end{enumerate}
\end{corollary}
\begin{remark}
    Corollary \ref{Infintely many expanders coming from the same cone} confirmed a conjecture raised by Chen \cite[Section 7]{Chen2023Rotational} that for the rotationally symmetric triple cone $\Ccal_{\epsilon}$ defined above, one cannot expect the self-expander asymptotic to $\Ccal_{\epsilon}$ to be always rotationally symmetric.
\end{remark}

By analyzing the behavior of $\{ \Sigma[g]\}_{g \in \mathbb{N}}$ as $g \to \infty$, we are able to characterize $\Sigma[g]$ when $g \in \mathbb{N}$ is sufficiently large.

\begin{proposition}[Proposition \ref{Behavior for high genus}] 
For any $\epsilon > \delta_0$, the sequence of self-expanders $\{ \Sigma[g, \epsilon]\}_{g \in \mathbb{N}}$ constructed in Corollary \ref{Infintely many expanders coming from the same cone} (by passing to a subsequence) converges to the union of the hyperplane through the origin and a self-expander annulus that is asymptotic to $\Ccal(\epsilon)$. The convergence is locally smooth away from the intersecting circle and is of multiplicity $1$.
\end{proposition}

\begin{remark}
Note that by our construction of the shrinker $\breve{M}_g$ from gluing PDE methods, the asymptotic cone $\breve{C}_g$ of $\breve{M}_g$, is very close to a rotationally symmetric cone $\Ccal_{\epsilon_0}$ for some $\epsilon_0 > \delta_0$ when $g \in \mathbb{N}$ becomes sufficiently large (\cite[(8.3)]{GZ1}), but can never coincide with $\Ccal_{\epsilon_0}$ as there exists no smooth properly embedded shrinker whose asymptotic cone is rotationally symmetric by \cite[Corollary 1.4]{Wang2014Uniqueness}. 

This minor difference between $\breve{C}_g$ and the rotationally symmetric cone $\Ccal_{\epsilon_0}$ shows the delicacy of Conjecture \ref{Genus Reduction Conjecture}: if it happened that $\breve{C}_g = \Ccal_{\epsilon_0}$ for some $g \in \mathbb{N}$ and $\epsilon_0 > \delta_0$ sufficiently large, then by Corollary \ref{Infintely many expanders coming from the same cone} and \ref{Genus Drop example of a MCF}, we could construct an MCF where its genus can increase at the first singular time, which violates Conjecture \ref{Genus Reduction Conjecture}.
\end{remark}

As a result of Corollary \ref{Infintely many expanders coming from the same cone}, we can fully characterize the level set flow of $\Ccal_{\epsilon}$ for $\epsilon > \delta_0$.
\subsection*{Fattening of $\Ccal_{\epsilon}$}
As a minimal surface in the conformal metric \eqref{eq:gexp}, any $\Sigma[g]$ constructed in Corollary \ref{Infintely many expanders coming from the same cone} is unstable by \cite[Lemma 8.3]{Bernstein2023Integer}. Now by the nature of unstable minimal surfaces, we shall be able to connect $\Sigma[g]$ to some stable minimal surface using the negative gradient flow of the functional $E$. More precisely, we consider the \emph{expander mean curvature flow} (EMCF).
 \begin{equation}
     \left(\frac{\partial X_{\Sigma_t}}{\partial t}\right)^{\perp} = H_{\Sigma_t}\nu_{\Sigma_t} - \frac{X_{\Sigma_t}^{\perp}}{2},
 \end{equation}
 where $t \mapsto \Sigma_t \subset \mathbb{R}^3$ is a family of surfaces. Bernstein-Wang-Chen \cite{bernstein2024existence} have studied EMCF and developed a full machinery to produce stable self-expanders from unstable ones by using this flow:
  \begin{theorem}[{\cite[Theorem 1.3]{bernstein2024existence}}] \label{Morse line flow producing stable self-expanders from unstable ones} For $2 \leq n \leq 6$, let $\Ccal$ be a $C^3$-regular cone in $\mathbb{R}^{n + 1}$. Suppose $\Omega_- \subset \mathbb{R}^{n + 1}$ is a closed set and $\Sigma_- = \partial \Omega_-$ is a smooth self-expander asymptotic to $\Ccal$. If $\Sigma_-$ is strictly unstable, i.e., each component is unstable, then there exists a stable smooth self-expander $\Sigma_+$ asymptotic to $\Ccal$ and a strictly monotone expander Morse flow line $\Omega$ from $\Sigma_-$ to $\Sigma_+$ with $\Omega_- = cl(\bigcup_{t \in \mathbb{R}}\Omega(t))$. Moreover, if $\Omega^{\prime} \subset \mathbb{R}^{n + 1}$ is a closed set and $\partial \Omega^{\prime}$ is a smooth self-expander asymptotic to a $C^3$-regular cone $\Ccal^{\prime}$, and $\Omega^{\prime} \subset int(\Omega_-)$, then $\Omega' \subset \bigcap_{t \in \mathbb{R}}\Omega(t)$.
 \end{theorem}
 
The surface $\Sigma[g]$ constructed in Corollary \ref{Infintely many expanders coming from the same cone} divides $\mathbb{R}^3$ into 2 components $\Omega_-^g,\mathbb{R}^3 \setminus int(\Omega_-^g)$ such that $\partial \Omega_-^g = \partial (\mathbb{R}^3 \backslash int(\Omega_-^g)) = \Sigma[g]$. Applying Theorem \ref{Morse line flow producing stable self-expanders from unstable ones} above to both $\Omega_-^g$ and $\mathbb{R}^3 \backslash int(\Omega_-^g)$, we obtain 2 disjoint stable self-expanders $\Sigma[g]^1, \Sigma[g]^2$ and corresponding disjoint strictly monotone expander Morse flow lines from $\Sigma[g]$ to  $\Sigma[g]^1, \Sigma[g]^2$ (see \cite[Definition 1.2]{bernstein2024existence} for the definition of strictly monotone expander Morse flow line). Moreover, $\Sigma[g]^1, \Sigma[g]^2$ are both asymptotic to $\Ccal_{\epsilon}$. By the definition of the strictly monotone Morse flow lines and the geometry of $\Sigma[g]$, each $\Sigma[g]^i$ has 2 components and their union bounds a region that contains the cone $\Ccal_{\epsilon}$ (see Figure \ref{fig:Morse flow line picture 1} and \ref{fig:Morse flow line picture 2} below). By \cite[Lemma 5]{angenent1995computed}, we can characterize the level set flow of $\Ccal_{\epsilon}$.
\begin{figure}
\begin{tikzpicture}
    \draw[color = black, thick] (-5, 0) -- (1, 0);
    \draw[color = black, thick] (-0.5, -1.4) .. controls (-1.5, 0) .. (-0.5, 1.4);
    \draw[color = black, thick] (-3.5, -1.4) .. controls (-2.5, 0) .. (-3.5, 1.4);
    \draw[color = black, very thick] (-1.1, 0.2) -- (-0.1, 1.2);
    \draw[color = black, very thick] (-0.9, 0.1) -- (0.1, 1.1);
    \draw[color = black, very thick] (-0.6, 0.1) -- (.5, 1.2);
    \draw[color = black, very thick] (-0.3, 0.1) -- (0.8, 1.2);
    \draw[color = black, very thick] (0, 0.1) -- (1.3, 1.4);
    \draw[color = black, very thick] (0.3, 0.1) -- (1.7, 1.5);
    \draw[color = black, very thick] (.6, 0.1) -- (2.1, 1.6);
    \draw[color = black, very thick] (-2.9, 0.2) -- (-3.9, 1.2);
    \draw[color = black, very thick] (-3.1, 0.1) -- (-4.1, 1.1);
    \draw[color = black, very thick] (-3.4, 0.1) -- (-4.5, 1.2);
    \draw[color = black, very thick] (-3.7, 0.1) -- (-4.9, 1.3);
    \draw[color = black, very thick] (-4, 0.1) -- (-5.3, 1.4);
    \draw[color = black, very thick] (-4.3, 0.1) -- (-5.7, 1.5);
    \draw[color = black, very thick] (-4.6, 0.1) -- (-6, 1.6);
     \draw[color = black, very thick] (-.6, -1.2) -- (-2, -2.6);
    \draw[color = black, very thick] (-.7, -1.1) -- (-2, -2.4);
    \draw[color = black, very thick] (-.8, -1) -- (-2, -2.2);
    \draw[color = black, very thick] (-.9, -.9) -- (-2, -2);
    \draw[color = black, very thick] (-1, -0.8) -- (-2.1, -1.9);
    \draw[color = black, very thick] (-1.1, -0.7) -- (-2.2, -1.8);
    \draw[color = black, very thick] (-1.1, -0.5) -- (-2.3, -1.7);
    \draw[color = black, very thick] (-1.2, -0.3) -- (-2.4, -1.5);
    \draw[color = black, very thick] (-1.3, -0.2) -- (-2.5, -1.4);
    \draw[color = black, very thick] (-1.4, -0.1) -- (-2.8, -1.5);
    \draw[color = black, very thick] (-1.6, -0.1) -- (-2.8, -1.3);
    \draw[color = black, very thick] (-1.8, -0.1) -- (-3, -1.3);
    \draw[color = black, very thick] (-2, -0.1) -- (-3.3, -1.4);
    \draw[color = black, very thick] (-2.2, -0.1) -- (-3.2, -1.1);
    \draw[color = black, very thick] (-2.4, -0.1) -- (-2.8, -0.5);
    \node at (-2.5, -1.8) {$\Omega_-^g$};
    \node at (1.1, 0.2) {$\Omega_-^g$};
    \node at (-5.1, 0.2) {$\Omega_-^g$};
    \node at (-.5, 1.5){$\Sigma[g]$};
    \draw[->] (1.5,0) -- (2.5,0);
    \filldraw[color=black!60, fill=black!5, very thick](7.121, 0.1) .. controls (5.27717, 0.11481) .. (6.5, 1.4);
    \filldraw[color=black!60, fill=black!5, very thick](2.779, 0.1) .. controls (4.72283, 0.11481) .. (3.5, 1.4);
    \filldraw[color=black!60, fill=black!5, very thick](6.2, -1.2) .. controls (5, -0.05) .. (3.8, -1.2);
    \node at (6.5, 1.6) {$\Sigma[g]^1$};
    \node at (3.5, 1.6) {$\Sigma[g]^1$};
    \node at (6.2, -1.4) {$\Sigma[g]^1$};
\end{tikzpicture}
\caption{The expander Morse flow line from $\Sigma[g]$ to $\Sigma[g]^1$.}
\label{fig:Morse flow line picture 1}
\end{figure}
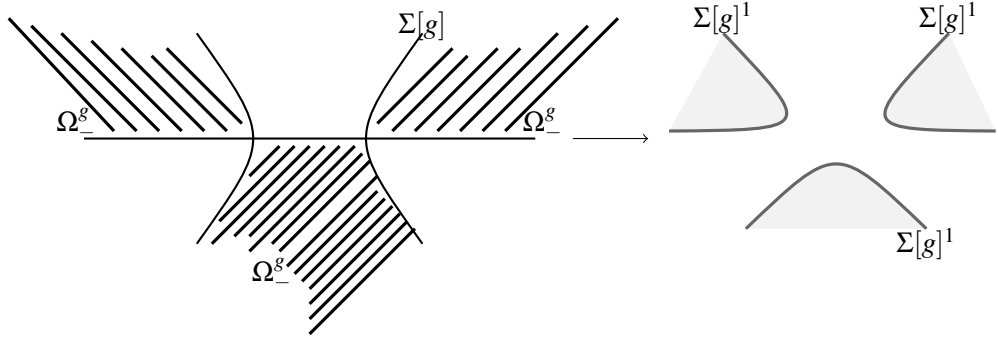
\begin{figure}
\begin{tikzpicture}
    \draw[color = black, thick] (-5, 0) -- (1, 0);
    \draw[color = black, thick] (-0.5, -1.4) .. controls (-1.5, 0) .. (-0.5, 1.4);
    \draw[color = black, thick] (-3.5, -1.4) .. controls (-2.5, 0) .. (-3.5, 1.4);
    \draw[color = black, very thick] (-1.1, -0.2) -- (-0.1, -1.2);
    \draw[color = black, very thick] (-0.9, -0.1) -- (0.1, -1.1);
    \draw[color = black, very thick] (-0.6, -0.1) -- (.5, -1.2);
    \draw[color = black, very thick] (-0.3, -0.1) -- (0.8, -1.2);
    \draw[color = black, very thick] (0, -0.1) -- (1.3, -1.4);
    \draw[color = black, very thick] (0.3, -0.1) -- (1.7, -1.5);
    \draw[color = black, very thick] (.6, -0.1) -- (2.1, -1.6);
    \draw[color = black, very thick] (-2.9, -0.2) -- (-3.9, -1.2);
    \draw[color = black, very thick] (-3.1, -0.1) -- (-4.1, -1.1);
    \draw[color = black, very thick] (-3.4, -0.1) -- (-4.5, -1.2);
    \draw[color = black, very thick] (-3.7, -0.1) -- (-4.9, -1.3);
    \draw[color = black, very thick] (-4, -0.1) -- (-5.3, -1.4);
    \draw[color = black, very thick] (-4.3, -0.1) -- (-5.7, -1.5);
    \draw[color = black, very thick] (-4.6, -0.1) -- (-6, -1.6);
    \draw[color = black, very thick] (-.6, 1.2) -- (-2, 2.6);
    \draw[color = black, very thick] (-.7, 1.1) -- (-2, 2.4);
    \draw[color = black, very thick] (-.8, 1) -- (-2, 2.2);
    \draw[color = black, very thick] (-.9, .9) -- (-2, 2);
    \draw[color = black, very thick] (-1, 0.8) -- (-2.1, 1.9);
    \draw[color = black, very thick] (-1.1, 0.7) -- (-2.2, 1.8);
    \draw[color = black, very thick] (-1.1, 0.5) -- (-2.3, 1.7);
    \draw[color = black, very thick] (-1.2, 0.3) -- (-2.4, 1.5);
    \draw[color = black, very thick] (-1.3, 0.2) -- (-2.5, 1.4);
    \draw[color = black, very thick] (-1.4, 0.1) -- (-2.8, 1.5);
    \draw[color = black, very thick] (-1.6, 0.1) -- (-2.8, 1.3);
    \draw[color = black, very thick] (-1.8, 0.1) -- (-3, 1.3);
    \draw[color = black, very thick] (-2, 0.1) -- (-3.3, 1.4);
    \draw[color = black, very thick] (-2.2, 0.1) -- (-3.2, 1.1);
    \draw[color = black, very thick] (-2.4, 0.1) -- (-2.8, 0.5);
    \node at (-3, 2.1) {$\mathbb{R}^3 \backslash int(\Omega_-^g)$};
    \node at (1.1, -1.8) {$\mathbb{R}^3 \backslash int(\Omega_-^g)$};
    \node at (-5.1, -1.8) {$\mathbb{R}^3 \backslash int(\Omega_-^g)$};
    \node at (-.5, 1.5){$\Sigma[g]$};
    \draw[->] (1.5,0) -- (2.5,0);
    \filldraw[color=black!60, fill=black!5, very thick](7.121, -0.1) .. controls (5.27717, -0.11481) .. (6.5, -1.4);
    \filldraw[color=black!60, fill=black!5, very thick](2.779, -0.1) .. controls (4.72283, -0.11481) .. (3.5, -1.4);
    \filldraw[color=black!60, fill=black!5, very thick](6.2, 1.2) .. controls (5, 0.05) .. (3.8, 1.2);
    \node at (6.5, -1.6) {$\Sigma[g]^2$};
    \node at (3.5, -1.6) {$\Sigma[g]^2$};
    \node at (6.2, 1.4) {$\Sigma[g]^2$};
\end{tikzpicture}
\caption{The expander Morse flow line from $\Sigma[g]$ to $\Sigma[g]^2$.}
\label{fig:Morse flow line picture 2}
\end{figure}
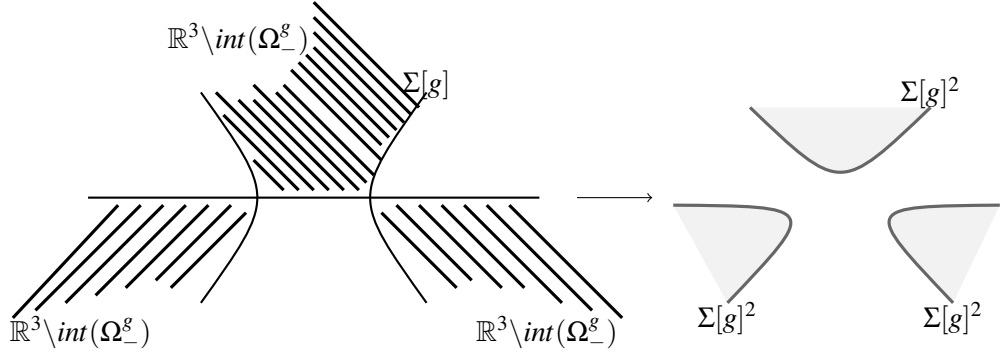
\begin{corollary}
For any $\epsilon > \delta_0$, let $E^1, \dots, E^4$ be the 4 connected components of $\mathbb{R}^3 \backslash \Ccal_{\epsilon}$ and  $\{\Gamma_t \}_{t \geq 0}$ be the level set flow of $\Ccal_{\epsilon}$. Then $\Gamma_t$ also divides $\mathbb{R}^3$ into 4 connected components $E_t^1, \dots, E_t^4$, where $\{ \mathbb{R}^3 \backslash E_t^i\}_{t \geq 0}$ is a level set flow for each $i \in \{1, \dots, 4\}$. Moreover, $\{\Gamma_t \}_{t \geq 0}$ satisfies the following.
\begin{enumerate}[(i)]
\item For all $g \in \mathbb{N}$, $\sqrt{t} \cdot \Sigma[g] \subset int(\Gamma_t), \sqrt{t} \cdot (\Sigma[g]^1 \cup \Sigma[g]^2) \subset \Gamma_t$ for $t > 0$.
\item $\partial \Gamma_1$ is a smooth stable self-expander.
\end{enumerate}  
\end{corollary}
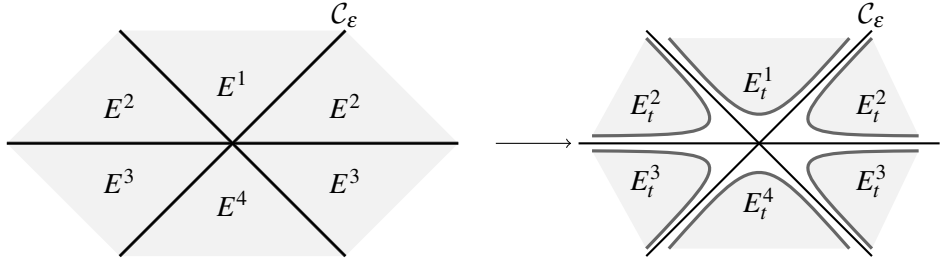
\begin{figure}
    \centering
    \begin{tikzpicture}
    \filldraw[color=black!60, fill=black!5, very thick](-5, 0) --(-2,0) -- (-3.5, 1.5);
    \filldraw[color=black!60, fill=black!5, very thick](1, 0) --(-2,0) -- (-.5, 1.5);
    \filldraw[color=black!60, fill=black!5, very thick](-3.5, 1.5) --(-2,0) -- (-.5, 1.5);
    \filldraw[color=black!60, fill=black!5, very thick](-3.5, -1.5) --(-2,0) -- (-.5, -1.5);
    \filldraw[color=black!60, fill=black!5, very thick](-5, 0) --(-2,0) -- (-3.5, -1.5);
    \filldraw[color=black!60, fill=black!5, very thick](1, 0) --(-2,0) -- (-.5, -1.5);
    \draw[color = black, thick] (-5, 0) -- (1, 0);
    \draw[color = black, thick] (-3.5, -1.5) -- (-0.5, 1.5);
    \draw[color = black, thick] (-3.5, 1.5) -- (-0.5, -1.5);
    \draw[->] (1.5,0) -- (2.5,0);
    \draw[color = black, thick] (2.6, 0) -- (7.4, 0);
    \draw[color = black, thick] (3.5, -1.5) -- (6.5, 1.5);
    \draw[color = black, thick] (3.5, 1.5) -- (6.5, -1.5);
    \filldraw[color=black!60, fill=black!5, very thick](7.121, -0.1) .. controls (5.27717, -0.11481) .. (6.5, -1.4);
    \filldraw[color=black!60, fill=black!5, very thick](2.779, -0.1) .. controls (4.72283, -0.11481) .. (3.5, -1.4);
    \filldraw[color=black!60, fill=black!5, very thick](6.2, 1.4) .. controls (5, 0.05) .. (3.8, 1.4);
    \filldraw[color=black!60, fill=black!5, very thick](7.121, 0.1) .. controls (5.27717, 0.11481) .. (6.5, 1.4);
    \filldraw[color=black!60, fill=black!5, very thick](2.779, 0.1) .. controls (4.72283, 0.11481) .. (3.5, 1.4);
    \filldraw[color=black!60, fill=black!5, very thick](6.2, -1.4) .. controls (5, -0.05) .. (3.8, -1.4);
    \node at (-.5, 1.7) {$\Ccal_{\epsilon}$};
    \node at (-2, 0.8) {$E^1$};
    \node at (-0.5, 0.5) {$E^2$};
    \node at (-3.5, 0.5) {$E^2$};
    \node at (-0.5, -0.5) {$E^3$};
    \node at (-3.5, -0.5) {$E^3$};
    \node at (-2, -0.8) {$E^4$};
    \node at (6.5, 1.7) {$\Ccal_{\epsilon}$};
    \node at (5, 0.8) {$E_t^1$};
    \node at (6.5, 0.5) {$E_t^2$};
    \node at (3.5, 0.5) {$E_t^2$};
    \node at (6.5, -0.5) {$E_t^3$};
    \node at (3.5, -0.5) {$E_t^3$};
    \node at (5, -0.8) {$E_t^4$};
    \end{tikzpicture}
    \caption{Fattening of $\Gamma_t$.}
    \label{fig:enter-label}
\end{figure}

\subsection*{Idea of the proof of Theorem \ref{the main result of the paper}}
 Our proof of Theorem \ref{the main result of the paper} is motivated by Ilmanen \cite[p.13-14]{ilmanen1998lectures} and Ding \cite[Theorem 6.3]{ding2020minimal}. We construct a self-expander with prescribed topology and boundary in each closed ball, and then extend it to a complete self-expander with the same topology by passing to a limit. 
 
 The main difficulty is that we have to carefully control the topology of the self-expanders which we constructed in each closed ball. Instead of considering the $E$-minimizing self-expanding current with prescribed boundary as \cite[p.13-14]{ilmanen1998lectures} and \cite[Theorem 6.3]{ding2020minimal}, we appeal to a result of Meeks-Simon-Yau \cite[Theorem 1]{meeks1982embedded} to minimize area among the isotopy class of the Costa-Hoffman-Meeks surface \cite{Hoffman1990Embedded}, which has $\Grp[g]$-symmetry and genus $g$. To prevent the minimizer from degenerating into 3 disks, we use a barrier argument to make sure if such degeneration occurs, the 3 disks are sufficiently close to each other. Thus, they will have large area accumulating inside the unit ball, by performing a small surgery and desingularization in the unit ball, we can not only recover the topology (connectedness and positive genus), but also decrease the area by a fixed amount. Now we can redo the minimization. After finitely many such iterations, we will arrive at the connected positive genus self-expander which we desire in the closed ball. 

 As the boundary of each expander we construct in the closed ball is prescribed by the cone, the barrier argument by \cite[Lemma 8.2]{Bernstein2023Integer} (also see \cite[Theorem 6.3]{ding2020minimal}) will force the limit self-expander to be asymptotic to the same cone.

\subsection*{Acknowledgments} GS wants to thank his advisor Daniel Ketover for suggesting this problem and his constant encouragement and support. Both authors are very grateful to Jacob Bernstein for several useful discussions and his interest in this work. The authors also would like to thank the interests and valuable comments of Ao Sun and Otis Chodosh.

\section{Preliminaries and notation}
\subsection*{Preliminaries} \label{Preliminaries}
In this subsection, we recall some basic properties of self-expanders used throughout the paper. Readers who are familiar with self-expanders may skip this section. 

On any self-expander hypersurface $\Sigma \subset \mathbb{R}^{3}$ with the embedding $X = X_{\Sigma}$, 
we have
\[
\Delta_{\Sigma} X = H_{\Sigma}\nu_{\Sigma}=\frac{1}{2}X^{\perp},
\]
and thus
\begin{equation}\label{eq:Laplacian}
    \Delta_{\Sigma} |X|^2 = 2  X \cdot \Delta_{\Sigma} X + 2 |\nabla_{\Sigma} X|^2 = 2  X \cdot \Delta_{\Sigma} X + 4 = |X^{\perp}|^2 + 4 = 4|H_{\Sigma}|^2 + 4.
\end{equation}
Moreover, self-expanders, as minimal surfaces in the conformal metric \eqref{eq:gexp}, possesses the \emph{monotonicity formula} (\cite[(2.7)]{ding2020minimal}):
\begin{equation} \label{Monotonicity Formula}
    \frac{d}{d\rho}\frac{\mathcal{H}^2(\Sigma \cap B_{\rho})}{\rho^{2}} = \frac{d}{d\rho} \int_{\Sigma \cap B_{\rho}} \frac{|X^{\perp}|^2}{|X|^{4}} \, d\mathcal{H}^2 + \frac{1}{2\rho^{3}} \int_{\Sigma \cap B_{\rho}} |X^{\perp}|^2 \, d\mathcal{H}^2 \geq 0
\end{equation}
for any $\rho > 0$, where $B_{\rho}$ denotes the ball in $\mathbb{R}^{3}$ with radius $\rho$ and centered at the origin. A self-expander $\Sigma$ is called \emph{stable} if $\Sigma$ is a stable minimal surface in the expander space \eqref{eq:gexp}, i.e., the second variation of $\Sigma$ with respect to the functional $E[\Sigma]$ in \eqref{eq:E} is nonnegative.

Let us recall some classical definitions of currents \cite[Chapter 6, 7]{simon1984lectures}.
Given any open set $U \subset \mathbb{R}^3$, let $\Omega^k_{c}(U)$ be the set of all smooth $k$-forms having compact support in $U$. We use $\mathcal{D}_k(U)$ to denote the set of all $k$-currents on $U$, which are continuous linear functionals on $\Omega^k_{c}(U)$.  For each open set $W \subset U$ and $T\in\mathcal{D}_k(U)$, one defines:
\[
\mathbf{M}_W(T) = \sup\{ T(\omega):|\omega| \leq 1, \omega \in \Omega^k_{c}(W)\}.
\]

If $\mathbf{M}_W(T) < \infty$ for any $W \Subset U$, then by Riesz representation theorem, there is a Radon measure $\mu_T$ on $U$ and $\mu_T$-measurable $k$-vector-valued function $\vec{T}$ on $U$ with $|\vec{T}| = 1$ $\mu_T$-a.e., such that $\forall \omega\in \Omega^k_{c}(U)$,
\[
T(\omega) = \int \langle\omega(x), \vec{T}(x) \rangle \, d\mu_T(x).
\]
 
We call $T$ an $E$-minimizing self-expanding current in $U \subset \mathbb{R}^3$, if $T \in \mathcal{D}_2(U)$ is an integer multiplicity locally rectifiable current in $U$, and $\mathbf{E}_W(T) \leq \mathbf{E}_W(T^{\prime})$ whenever $W \Subset U$, $\partial T = \partial T^{\prime}$, where $\mathbf{E}_W(T)$ is given by
\begin{equation}
    \mathbf{E}_W(T) : =\sup\{ \int \langle\omega(x), \vec{T}(x) \rangle e^{\frac{|x|^2}{4}}\, d\mu_T(x):\omega \in \mathcal{D}_2(W), |\omega| \leq 1\}.
\end{equation}

By standard geometric measure theory \cite[Lemma 34.1]{simon1984lectures}, for any integer multiplicity current $S \in \mathcal{D}_1(U)$ with spt$S$ compact and $\partial S = \emptyset$, there exists an $E$-minimizing self-expanding current $T$ in $U$. Moreover, if $S$ is a multiplicity 1 smooth simple closed curve, $T$ is given by a multiplicity 1 smoothly embedded self-expander (see\cite[Theorem 37.7]{simon1984lectures} or \cite[Theorem 27.3]{maggi2012sets}).

Now we give a small but useful lemma for self-expanders that we will use throughout the paper.

\begin{lemma} [{\cite{Ecker1989Mean}} {(cf. \cite[p. 2]{Bernstein2022Topological})}] \label{graphical self-expander}
    Let $\Ccal \subset \mathbb{R}^{n + 1}$ be any smooth embedded graphical hypercone. Then there is a unique self-expander $\Sigma$ asymptotic to $\Ccal$. Moreover, $\Sigma$ is graphical. If $\Ccal$ is mean-convex, then $\Sigma$ lies above $\Ccal$.
\end{lemma}

\begin{proof}
The existence of a graphical self-expander $\Sigma$ asymptotic to $\Ccal$ follows from \cite[Theorem 5.1]{Ecker1989Mean}. If $\Sigma^{\prime}$ is any self-expander asymptotic to $\Ccal$, then both $\{\sqrt{t} \cdot \Sigma\}_{t \geq 0}$ and $\{\sqrt{t} \cdot \Sigma^{\prime}\}_{t \geq 0}$ are solutions to the MCF starting from $\Ccal$. Since $\Ccal$ is a smooth graph, it follows from the uniqueness for uniformly parabolic equations on a short time interval that  $\Sigma = \Sigma^{\prime}$. 

If $\Ccal$ is mean-convex, from the strong maximum principle, the MCF $\{\sqrt{t} \cdot \Sigma\}_{t \geq 0}$ lies above $\Ccal$ for all times. Thus, $\Sigma$ lies above $\Ccal$. 
\end{proof}

\subsection*{Notation}
Throughout the paper, for notational convenience, we let $\mathcal{L}[\Ccal]$ denote the link of a cone $\Ccal \subset \mathbb{R}^3$. We say $\Ccal$ is a smooth embedded cone if $\mathcal{L}[\Ccal]$ is a regular cuve on $\mathbb{S}^2$. For any smooth embedded cone $\Ccal[g] \in \Pi[g, \delta]$, we have
 \begin{equation*}
     \Ccal[g]=\cup_{i=1}^3\Ccal_i[g],
 \end{equation*}
 where each $\Ccal_i[g]$ is a simple cone over $\R^2$. We then write $\mathcal{L}_i$ for the link of $\Ccal_i[g]$. We also abbreviate by using $g(\Sigma)$ for the genus of a surface $\Sigma$. ,  
 
Throughout the paper, $\mathcal{H}^2$ stands for the standard 2-dimensional Hausdorff measure in the Euclidean space and $\mathcal{H}^2_w$ stands for the 2-dimensional Hausdorff measure in the space $(\mathbb{R}^3,g_{Exp})$, i.e., for any set $A \subset \mathbb{R}^3$,
\[
\mathcal{H}^2_w(A) := \int_A e^{\frac{|X|^2}{4}} \, d\mathcal{H}^2.
\]
For any $p \in \mathbb{R}^3$ and finite group $G$ acting on $\mathbb{R}^3$, we define the \emph{isotropy subgroup (or stabilizer subgroup) $G_{p}\subset G$ with respect to $p$} by
\[
G_{p} := \{ g \in G : gp = p \}.
\]
We define the \emph{singular locus} $S=S(G)$ of the group action $G$ to be:
\[
S(G): = \{ p \in \mathbb{R}^3: G_{p} \neq \{ e\} \}.
\]
And we say a surface $\Sigma$ is $G$-equivariant if $\varphi(\Sigma) = \Sigma$ for all $\varphi \in G$.

For any $g \in \mathbb{N}$, $\Grp[g]$ is the \emph{$(g+1)$-prismatic} finite group defined by (cf. \cite[(2.25)-(2.26) and Definition 2.28]{GZ1})
\begin{equation}\label{eq:G}
    \Grp[g]:=\{\sigma_{v}[0], \sigma_{v}\left[\frac{\pi}{g+ 1}\right], \UUU\left[\frac{\pi}{2(g+ 1) }\right]\},
\end{equation}
where for $c\in\R$, $\sigma_v[c]$ is the reflection with respect to the vertical planes (we use cylindrical coordinates in $\mathbb{R}^3$)
\begin{equation}
    P_c := \{(r, c, z):r, z \in \mathbb{R}\}
\end{equation}
i.e.,
\begin{equation} \label{symmetry 1}
    \sigma_v[c](x, y, z) := (x \cos 2c + y \sin 2c, x \sin 2c - y \cos 2c, z));
\end{equation}
and $\UUU[c]$ is the reflection with respect to the horizontal lines 
\begin{equation*}
    l_c :=\{(r, c, 0): r \in \mathbb{R}\},
\end{equation*}
i.e.,
\begin{equation} \label{symmetry 2}
   \UUU[c](x, y, z) := (x \cos 2c + y \sin 2c, x \sin 2c - y \cos 2c, -z).
\end{equation}
Thus $\Grp[g]  \cong D_{2(g + 1)}$ where $D_{2(g + 1)}$ is the abstract dihedral group of order $2(g+1)$. Apparently it contains the rotation $\CCC\left[\frac{2\pi}{g+ 1}\right]$, where 
\begin{equation} \label{symmetry 3}
   \CCC[c](x, y, z) := (x \cos c - y \sin c, x \sin c + y \cos c, z).
\end{equation}
We also define the cyclic subgroup $\mathbf{C}_{g+1}$ of $\Grp[g]$ by
\begin{equation}\label{eq:cyclic}
    \mathbf{C}_{g+1}:=\{\CCC\left[\frac{2\pi}{g+ 1 }\right]\};
\end{equation}
and define the dihedral subgroup $\mathbf{D}_{g+1}$ of $\Grp[g]$ by
\begin{equation}\label{eq:dihedral}
    \mathbf{D}_{g+1}:=\{\CCC\left[\frac{2\pi}{g+ 1 }\right],\UUU\left[\frac{\pi}{2(g+ 1) }\right]\}.
\end{equation}

We also use $l_z$ to denote the $z$-axis.

We have the following simple lemma.
\begin{lemma}\label{lem:S}
    The singular locus $S[g]=S(\Grp[g])$ of $\Grp[g]$ is given by:
\[
 S [g]= S_0[g] \cup S_1 [g]\cup S_2[g], 
 \]
where $S_i$ for $i=0,1,2$ is defined by the following.
 \begin{enumerate}[(i)]
     \item $S_0[g]$ is the origin (Hausdorff dimension 0).
     \item $S_1[g] : =l_z \cup \bigcup_{i = 1}^{g + 1} l_{\frac{(2i-1)\pi}{2(g+1)}}$ is the union of the $z$-axis and $(g + 1)$ horizontal axes (Hausdorff dimension 1).
     \item $S_2[g] := \bigcup_{j = 0}^{g} P_{\frac{j \pi}{g + 1}}$ is the union of $(g + 1)$ vertical planes (Hausdorff dimension 2). 
\end{enumerate}
\end{lemma}


We use the notation $P_{xy}$ to denote the $xy$-plane in $\R^3$.
\begin{equation*}
    P_{xy} := \{ z = 0\}.
\end{equation*}
For $R>0$, we use $B_R$ to denote the ball in $\R^3$ centered at origin and with radius $R$.
And for $\delta>0$, we use $\Ccal(\delta)$ to denote the rotationally symmetric double cone  
\begin{equation}\label{eq:Cdelta}
    \Ccal(\delta): = \{(x,y,z):x^2 + y^2 = \delta z^2 \}.
\end{equation}
\section{Construction of barriers}
In this section, we will construct the barriers, which consist of  2 graphical complete self-expanders.
\begin{lemma}
There exist unique complete graphical self-expanders $\Lambda_1(\delta), \Lambda_2(\delta)$ asymptotic to $\Ccal(\delta) \cap \{z > 0\}, \Ccal(\delta) \cap \{z < 0\}$ respectively (recall \eqref{eq:Cdelta}). Moreover, the following hold.
\begin{enumerate}[(i)]
    \item $\Lambda_1(\delta)$ lies above $\Ccal(\delta)$. 
    \item $\Lambda_1(\delta)$ is rotationally symmetric with respect to $l_z$.
    \item When $\delta \to \infty$, $\Lambda_1(\delta)$ converges to the $xy$-plane $P_{xy} $ smoothly with multiplicity 1 as a graph.
    \item $\Lambda_2(\delta)$ is isometric to $\Lambda_1(\delta)$ via reflection with respect to the $xy$-plane $P_{xy} $.
\end{enumerate}
\end{lemma}
\begin{proof}
    From the construction of Angenent-Ilmanen-Chopp \cite{angenent1995computed} or Ecker-Huisken  \cite[Remark 5.2 iii)]{Ecker1989Mean}.
\end{proof}
\begin{notation}\label{not:Omegadelta}
    We use $\Omega(\delta)$ to denote the component of $\mathbb{R}^3 \backslash (\Lambda_1(\delta) \cup \Lambda_2(\delta))$ containing the $xy$-plane $P_{xy} $.
\end{notation}

$\Lambda_1(\delta)$ and $\Lambda_2(\delta)$ will be our barriers with a suitable choice of $\delta $ determined by the following technical lemma. 

\begin{lemma} \label{Unfiorm lower bound for minimal surfaces using barriers}
    For any $\epsilon \in (0, \mathcal{H}_w^2(P_{xy} \cap B_1))$, there exists $\delta_0 = \delta_0(\epsilon) > 0$ depending only on $\epsilon$ such that for any stable self-expander $\Sigma$ properly embedded in $B_1 \cap \Omega(\delta_0)$, $\Sigma$  has weighted area $\mathcal{H}_w^2(\Sigma) \geq \mathcal{H}_w^2(P_{xy} \cap B_1) - \epsilon$.
\end{lemma}
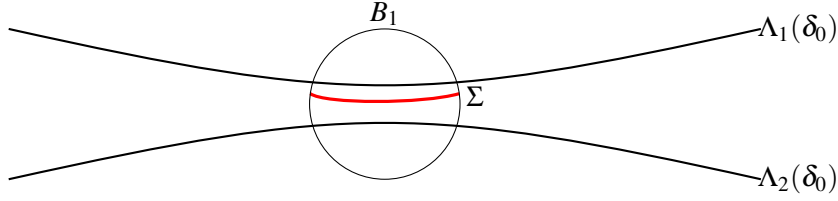
\begin{figure}
\centering
\begin{tikzpicture}
\draw[draw=black, thick] (-5,-1) .. controls (-.5,0) and (.5, 0)  .. (5,-1);
\draw [draw=red, very thick]  (-0.9903,0.1392) .. controls (-0.8,0) and (0.5, 0).. (0.9903,0.1392);
\node at (1.2, 0.1){$\Sigma$};
\node at (5.5, 1) {$\Lambda_1(\delta_0)$};
\node at (5.5, -1) {$\Lambda_2(\delta_0)$};
\draw[draw=black, thick] (-5,1) .. controls (-.5,0) and (.5,0) .. (5,1);
\draw (0,0) circle (1cm);
\node at (0, 1.2) {$B_1$};
\end{tikzpicture}
 \caption{$\Sigma$ in $B_1 \cap \Omega(\delta_0)$.}
  \label{fig:circle1}
\end{figure}
\begin{proof}
 We prove the lemma by contradiction. We fix $\epsilon > 0$. Suppose there exists a sequence of stable self-expanders $\Sigma_{\delta_n} \subset \Omega(\delta_n) \cap B_1$ properly embedded such that $ \mathcal{H}^2_w(\Sigma_{\delta_n}) < \mathcal{H}_w^2(P_{xy} \cap B_1) - \epsilon$ and $\delta_n \to \infty$. Now using $\Lambda_1(\delta_n) \cap B_1, \Lambda_2(\delta_n) \cap B_1$ as barriers and the compactness theorem for stable minimal surfaces in  \cite{schoen1983estimates},
 we can show, after passing to a subsequence, that $\Sigma_{\delta_n}$ converges to $P_{xy}$ smoothly with some multiplicity $\eta$. Hence $\lim_{n \to \infty} \mathcal{H}^2_w(\Sigma_{\delta_n}) = \eta \mathcal{H}^2_w(P_{xy} \cap B_1) \geq \mathcal{H}^2_w(P_{xy} \cap B_1)$. This implies that $\mathcal{H}_w^2(\Sigma_{\delta_n}) \geq \mathcal{H}_w^2(P_{xy} \cap B_1) - \epsilon$ for all $n \in \mathbb{N}$ sufficiently large, which provides a contradiction.
\end{proof}

\begin{assumption} \label{Annulus has arbitr aily small area}
By Lemma \ref{Unfiorm lower bound for minimal surfaces using barriers} above, we can choose $\delta_0 > 0$ large enough such that any properly embedded stable self-expander $\Sigma$ in $B_1 \cap \Omega(\delta_0)$ has weighted area
\begin{equation*}
    \mathcal{H}^2_w(\Sigma) \geq \frac{1}{2} \mathcal{H}^2_w(P_{xy} \cap B_1) = 2\pi(e^{\frac{1}{4}}-1).
\end{equation*}

Observe that as $\Lambda_1(\delta) \cap \partial B_1, \Lambda_2(\delta_0) \cap \partial B_1$ converges to the equator circle $\partial B_1 \cap \{z = 0\}$ as $\delta \to \infty$. By increasing $\delta_0 > 0$ if necessary, we may assume that the annulus $\Acal \subset \partial B_1$ bounded by  $\Lambda_1(\delta_0) \cap \partial B_1, \Lambda_2(\delta_0) \cap \partial B_1$ has weighted area smaller than $\frac{1}{4}\mathcal{H}^2_w(P_{xy} \cap B_1)$. 

For technical purposes, we also assume that the points of intersection
\begin{equation*}
    \cup_{i=1,2}\Lambda_i(\delta_0)\cap l_z,
\end{equation*}
are contained in the unit ball $B_1$ by increasing $\delta_0 > 0$ if necessary. 
\end{assumption}
$\delta_0 > 0$ defined in Assumption \ref{Annulus has arbitr aily small area} above is the universal constant we need in Theorem \ref{the main result of the paper}. From now on, we will fix this choice of $\delta_0$ and $\Lambda_1(\delta_0), \Lambda_2(\delta_0)$ will serve as our barriers. They will help us construct connected self-expanders with positive genus in each closed ball. Once we finish this step, a slight modification of the theorems by Ding \cite[Theorem 6.3]{ding2020minimal} and Bernstein-Wang \cite[Lemma 8.2]{Bernstein2023Integer} will give a complete self-expander in the whole space with the desired asymptotical behavior. For the future use, we record this theorem below and include the proof in the appendix \ref{appendix theorem for the convergence to the limit complete self-expander}.
\begin{proposition}\label{Convergence to the limit self-expander}
Let $\Ccal \subset \mathbb{R}^3$ be a smoothly embedded cone. Suppose that for all sufficiently $R > 0$, there exists a properly embedded self-expander $\Sigma_R \subset B_R$ satisfying the following.
\begin{enumerate}[(i)]
    \item $\Sigma_R$ is connected.
    \item $g(\Sigma_R)$ is uniformly bounded above by $C$. 
    \item $\partial \Sigma_R = R \mathcal{L}[\Ccal]$.
\end{enumerate}
Then there is a complete embedded self-expander $\Sigma_{\infty} \subset \mathbb{R}^3$ such that after passing to subsequences, $\Sigma_R$ converges to $\Sigma_{\infty}$ smoothly with multiplicity 1, and $\Sigma_{\infty}$ is asymptotic to $\Ccal$. Moreover, the following hold.
\begin{enumerate}[(i)]
    \item $\Sigma_{\infty}$ is connected.
    \item $g(\Sigma_{\infty}) \leq C$.
    \item $\Sigma_{\infty}$ has the quadratic volume growth: $\mathcal{H}^2(\Sigma_{\infty} \cap B_R) \leq \frac{\mathcal{H}^1(\mathcal{L}[\Ccal])}{2} R^2$ for any $R > 0$.
\end{enumerate}
\end{proposition}
We also give below an easy corollary of Proposition \ref{Convergence to the limit self-expander} for later uses.
\begin{corollary} \label{disk}
Let $\Ccal$ be a graphical cone. Suppose for all $R > 0$ sufficiently large,  there is a self-expanders $\Sigma_R$ such that the following hold:
\begin{enumerate}[(i)]
    \item $\partial \Sigma_R = R \mathcal{L}[\Ccal]$.
    \item $g(\Sigma_R) \leq C$ is uniformly bounded above.
\end{enumerate}
Then after passing to subsequences, $\{\Sigma_R \}_{R > 0}$ will converge to a complete graphical self-expander $\Sigma_{\infty}$ smoothly of multiplicity $1$. Moreover, $\Sigma_R \cap B_1$ is a topological disk for all $R > 0$ large.
\end{corollary}
\begin{proof}
By Proposition \ref{Convergence to the limit self-expander}, $\Sigma_R$ will converge smoothly to a complete self-expander $\Sigma_{\infty}$ that is asymptotic to $\Ccal$  of multiplicity $1$. By Lemma \ref{graphical self-expander}, $\Sigma_{\infty}$ is graphical. That $\Sigma_R \cap B_1$ is a topological disk then follows from the definition of smooth convergence of multiplicity $1$ .
\end{proof}

\section{Connected self-expanders with positive genus in closed balls}
\noindent We will use the result of Meeks-Simon-Yau \cite{meeks1982embedded} to minimize area in the isotopy class of some specific initial surface with the topology that we want. A natural choice of this initial surface will be any surface isotopic to the Costa-Hoffman-Meeks surface from \cite{Hoffman1990Embedded}, which has $\Grp[g]$-symmetry and genus exactly $g$. 
\begin{notation}\label{not:Omega}
For notational convenience, we abbreviate $\Omega$ for $\Omega(\delta_0)$, $\Pi[g]$ for $\Pi[g, \delta_0]$ (recall \ref{def:Pi}) from now on. We will also omit the parameter "$[g]$" in this section when it is clear from the context. 
\end{notation}
Fix $g \in \mathbb{N}$ and $\Ccal[g] \in \Pi[g]$. Then $\Ccal[g] \subset \{x^2 + y^2 > \delta_0 z^2\} \subset \Omega$. $\Omega \cap B_R$ will be our desired ambient space to apply an equivariant version of  Meeks-Simon-Yau \cite[Theorem 1]{meeks1982embedded} (also see  \cite[Proposition 3.3]{de2010genus}) which we state below. For completeness, we include the proof of this proposition in \ref{prop:2}.
\begin{proposition}\label{Equivariant version of the minimzing problem} 
   Let $\Sigma \subset \Omega \cap B_R$ be a properly embedded $\Grp$-equivariant  surface with smooth boundary $\partial  \Sigma$ intersecting $S \cap \partial B_R$ transversally. Suppose that $\{\Delta^j := \varphi^j(1, \Sigma)\}_j$ is a $\Grp$-equivariant minimizing sequence (see \ref{def:Gmini}) for Problem $(\Sigma, \mathbf{Js}_{\Grp}(\Omega \cap B_R))$ (see \ref{G-isotopies}) converging to a stationary varifold $V$. Then $V = \bigcup_{k = 1}^N \Gamma_k$ is a union of connected disjoint embedded minimal surfaces $\Gamma_1, \dots, \Gamma_N \subset B_R$ with respect to the metric $g_{Exp}$, each with multiplicity 1. Moreover, the following hold.
    \begin{enumerate}[(i)]
     \item $\partial \Gamma_1 \cup \dots \cup \partial\Gamma_k = \partial \Sigma$.
     \item $\sum\limits_{k = 1}^N g(\Gamma_k) \leq g(\Sigma)$.
    \end{enumerate}
\end{proposition}




\begin{figure}\label{fig:Scherk}
  \centering
  \includegraphics[width=0.4\textwidth]{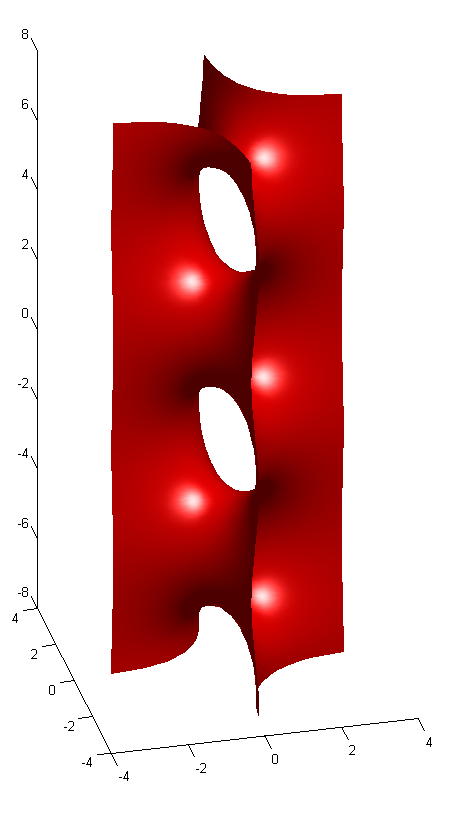}
  \caption{Scherk's Second Surface \cite{wikipedia-scherk}.}
\end{figure}

Now we give the main proposition in this section.
\begin{proposition}\label{prop:SigmaR}
For any $R > 0$ big enough, there is a smooth embedded self-expander $\Sigma_R \subset \Omega \cap B_R$ with boundary $R\mathcal{L}[\Ccal[g]]$ and genus $g$. 
\end{proposition}
 \begin{proof}
 We first choose a $\Grp[g]$-equivariant properly embedded annulus $\Acal_R\subset B_R \cap \Omega$ such that $\partial\Acal_R=R\mathcal{L}_1 \cup R\mathcal{L}_3$. We then choose a $\Grp[g]$-equivariant properly embedded disk $D_R \subset B_R \cap \Omega$ such that  $\partial D_R=R\mathcal{L}_2$ and $c_R:=\Acal_R\cap D_R$ is one simple closed curve in the interior of $D_R$. 

     Now we desingularize $\Acal_R \cup D_R$ in a small tubular neighborhood $T$ of $c_R$ by gluing the singular surface $(\Acal_R\cup D_R) \setminus T$ with a $\Grp[g]$-equivariant genus $g$ surface which is obtained after "bending" a $(g+1)$ period of \emph{the Scherk's Second Surface} (Figure \ref{fig:Scherk}) as in \cite[Definition 3.6]{kapouleas:1997}. We denote the resulting $\Grp[g]$-equivariant properly embedded genus-$g$ surface by $CHM_{R}$.

Applying Proposition \ref{Equivariant version of the minimzing problem} to $CHM_R$ in $\Omega \cap B_R$, we obtain a multiplicity $1$ embedded $\Grp[g]$-equivariant self-expander 
     \begin{equation}\label{eq:Sigma0}
         \Sigma_{R,0} \subset B_R \cap \Omega
     \end{equation}
     with boundary $R\mathcal{L}[\Ccal]$. Moreover, we have
\begin{equation} \label{genus bounds}
    g(\Sigma_{R, 0}) \leq g \text{ and } \mathcal{H}^2_w(\Sigma_{R, 0})=\inf_{\varphi \in \mathbf{Js}_{\Grp[g]}(\Omega \cap B_R)} \mathcal{H}^2_w(\varphi(1, CHM_{R})).
\end{equation}
 Due to the fact that $\partial\Sigma_{R,0}$ has 3 boundary components and the $\Grp[g]$-symmetry, $\Sigma_{R,0}$ either is connected  or has 3 components $C_{1, R}, C_{2,R}, C_{3, R}$ with boundary $R\mathcal{L}_1, R\mathcal{L}_2, R\mathcal{L}_3$ respectively.
     
If $\Sigma_{R,0}$ is connected, then an argument by Riemann-Hurwitz formula below will show that $g(\Sigma_{R, 0}) = g$
, and we just choose $\Sigma_R:=\Sigma_{R,0}$. 

Suppose $\Sigma_{R,0}$ is not connected. By the symmetry again and \cite[Lemma 3.4 and Lemma 3.5]{ketover2016equivariant}, $C_{2, R}$ must contain the horizontal axes $\bigcup_{j = 1}^{g + 1} l_{\frac{(2j-1)\pi}{2(g+1)}}$, orthogonal to $l_z$ (z-axis) and $\bigcup_{j = 0}^{g} P_{\frac{j \pi}{g + 1}}$; $C_{1, R}$ (isometric to $C_{3, R}$) is orthogonal to $l_z$ and $\bigcup_{j = 0}^{g} P_{\frac{j \pi}{g + 1}}$,  disjoint from the horizontal axes $\bigcup_{j = 1}^{g + 1} l_{\frac{(2j-1)\pi}{2(g+1)}}$. By \cite[Lemma B.1]{carlotto2020free}, $C_{i, R}$ has genus either 0 or $g$ for $i \in \{1, 2, 3\}$. Due to the genus bound (\ref{genus bounds}), $C_{1, R}, C_{3, R}$ can only be topological disks. Also note that since $\Ccal_i[g]$ is graphical, by Corollary \ref{disk}, $C_{i, R} \cap B_1$ is a topological disk for each $i \in \{1, 2, 3\}$. We define $\sigma_{i, R}:=\partial (C_{i, R} \cap B_1)$ to be the simple closed curve where each of them intersects the unit ball. 

Now we claim that $C_{2, R}$ is also a  topological disk for all $R > 1$ sufficiently large. Suppose on the contrary, there exists a sequence $\{R_n\}$ such that $R_n \to \infty$ and $g(C_{2, R_n}) \equiv g$. As $C_{2, R_n} \cap B_1$ is a topological disk by arguments above, $C_{2, R_n} \backslash B_1$ has 2 boundary components and genus $g$. Its Euler characteristic is then $\chi(C_{2, R_n} \backslash B_1) = -2g$. On the other hand, by Assumption \ref{Annulus has arbitr aily small area}, $\Lambda_1(\delta_0) \cap l_z \subset B_1$. Thus, $C_{2, R_n} \backslash B_1$ is disjoint from $l_z$. Then the quotient surface $(C_{2, R_n} \backslash B_1) / \mathbf{C}_{g+1}$ (recall \eqref{eq:cyclic}) is a compact Lipschitz surface with 2 boundary components and no branch points. Its Euler characteristic is then $\chi((C_{2, R_n} \backslash B_1) / \mathbf{C}_{g+1}) = -2 g((C_{2, R_n} \backslash B_1) / \mathbf{C}_{g+1})$. Now the action of $\mathbf{C}_{g+1}$ on $C_{2, R_n}$ is orientation-preserving, by the \textit{Riemann-Hurwitz formula}, 
we have
\begin{equation*}
\begin{split}
    -2g & =
    \chi(C_{2, R_n} \backslash B_1) \\
    & = (g + 1)\chi((C_{2, R_n} \backslash B_1) /\mathbf{C}_{g+1}))\\
    & = (g + 1)(-2g((C_{2, R_n} \backslash B_1) / \mathbf{C}_{g+1})),
\end{split}
\end{equation*}
which implies
\[
 g((C_{2, R_n} \backslash B_1) / \mathbf{C}_{g+1}) = \frac{g}{g + 1} \notin \mathbb{Z}.
\]
This is a contradiction. 
Therefore, $C_{1, R}, C_{2, R}, C_{3, R}$ are all topological disks.

 Now we perform a surgery on $C_{1, R} \cup C_{3, R}$. First observe that since they intersect transversally with $S[g]$ (recall \ref{lem:S}), we can define a $\Grp$-equivariant normal vector field on each of them, by \cite[Proposition 4.6]{ketover2016equivariant}, $C_{1, R}, C_{3, R}$ are stable minimal surfaces with respect to $g_{Exp}$ and hence for $i =1, 3$, $\mathcal{H}^2_w(C_{i, R} \cap B_1) \geq \frac{1}{2} \mathcal{H}^2_w(P_{xy} \cap B_1)$ holds by Assumption \ref{Annulus has arbitr aily small area}. We replace the two disks $C_{1, R} \cap B_1, C_{3, R} \cap B_1$ by the annulus $\Acal'_R \subset \partial B_1$ bounded by $\sigma_{1, R}, \sigma_{3, R}$. Observe that $\mathcal{H}_w^2(\Acal'_R) \leq \mathcal{H}_w^2(\Acal) < \frac{1}{4}\mathcal{H}^2_w(P_{xy} \cap B_1)$ since $\Acal'_R \subset \Acal$, where $\Acal$ is defined above in Assumption \ref{Annulus has arbitr aily small area}. After this process, we obtain a big annulus $\Acal''_R:= ((C_{1, R} \cup C_{3, R}) \backslash B_1) \cup \Acal'_R$ which has a smaller weighted area than $C_{1, R} \cup C_{3, R}$ since
\begin{equation} \label{Area decrease by adding in the annulus and reomving the two disks}
\begin{split}
    \mathcal{H}_w^2(\Acal''_R) & = 2\mathcal{H}_w^2(C_{1,R})  - 2\mathcal{H}_w^2(C_{1,R} \cap B_1) + \mathcal{H}_w^2(\Acal'_R) \\
    & \leq 2\mathcal{H}_w^2(C_{1,R}) - \mathcal{H}_w^2(P_{xy} \cap B_1) + \frac{1}{4}\mathcal{H}^2_w(P_{xy} \cap B_1) \\
    & = 2\mathcal{H}_w^2(C_{1,R}) - \frac{3}{4}\mathcal{H}_w^2(P_{xy} \cap B_1), 
\end{split}
\end{equation}

Now observe that $\Acal''_R$ intersects $C_{2, R}$ only along the simple closed curve $\sigma_{2, R}$, by a similar desingularization above, we can desingularize $\Acal''_R \cup C_{2, R}$ along $\sigma_{2, R}$ to obtain a genus $g$ smoothly embedded $\Grp[g]$-equivariant surface $CHM_{R,1}$ (smoothing if necessary) that is isotopic to a part of the Costa-Hoffman-Meeks surface. Since the desingularization only happens in a small tubular neighborhood of $\sigma_{2, R}$ and the change of area is as small as one wants, by \eqref{Area decrease by adding in the annulus and reomving the two disks}, we can make
\begin{equation*}
    \mathcal{H}_w^2(CHM_{R,1}) \leq 2\mathcal{H}_w^2(C_{1,R}) + \mathcal{H}^2_w(C_{2, R})- \frac{1}{2}\mathcal{H}_w^2(P_{xy} \cap B_1).
\end{equation*}

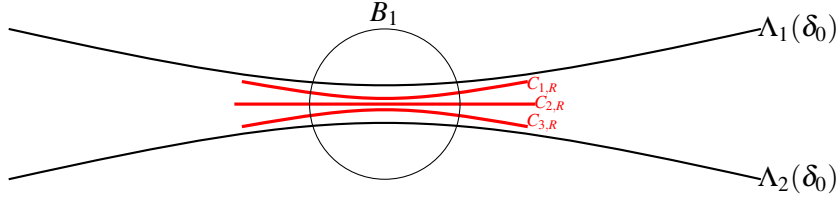
\begin{figure}
\centering
\begin{tikzpicture}
\draw [draw=red, very thick] (-2,0) -- (2,0);
\draw[draw=black, thick] (-5,-1) .. controls (-.5,0) and (.5, 0)  .. (5,-1);
\draw [draw=red, very thick]  (-1.9,0.3) .. controls (-.2,0) and (.2,0) .. (1.9,0.3);
\node[font=\tiny] at (2.1, 0.25){\textcolor{red}{$C_{1, R}$}};
\draw [draw=red, very thick]  (-1.9,-0.3) .. controls (-.2,0) and (.2,0) .. (1.9,-0.3);
\node[font=\tiny] at (2.1, -0.25){\textcolor{red}{$C_{3, R}$}};
\
\node[font=\tiny] at (2.2, 0){\textcolor{red}{$C_{2, R}$}};
\node at (5.5, 1) {$\Lambda_1(\delta_0)$};
\node at (5.5, -1) {$\Lambda_2(\delta_0)$};
\draw[draw=black, thick] (-5,1) .. controls (-.5,0) and (.5,0) .. (5,1);
\draw (0,0) circle (1cm);
\node at (0, 1.2) {$B_1$};
\end{tikzpicture}
 \caption{The 3 components $C_{1, R},C_{2, R},C_{3, R}$ if $\Sigma_R$ is not connected}
  \label{fig:circle2}
\end{figure}
\begin{figure}
\centering
\begin{tikzpicture}
\draw [draw=red, very thick] (-2,0) -- (2,0);
\draw [draw=red, ultra thick] (1,0) arc[start angle=0, end angle=8, radius=1cm];
\draw [draw=red, ultra thick] (1,0) arc[start angle=0, end angle=-8, radius=1cm];
\draw [draw=red, ultra thick] (-1,0) arc[start angle=0, end angle=8, radius=1cm];
\draw[draw=black, thick] (-5,-1) .. controls (-.5,0) and (.5, 0)  .. (5,-1);
\draw [draw=red, very thick]  (-1.9,0.3) .. controls (-1.1,.15)  .. (-1,0.1392);
\draw [draw=red, very thick]  (-1.9,-0.3) .. controls (-1.1,-.15)  .. (-1,-0.1392);
\draw [draw=red, very thick]  (1, 0.1392) .. controls (1.1,.15)  .. (1.9, 0.3);
\draw [draw=red, very thick]  (1, -0.1392) .. controls (1.1,-.15)  .. (1.9, -0.3);
\draw [draw=red, ultra thick] (-1,0) arc[start angle=172, end angle=180, radius=1cm];
\node[font=\tiny] at (2.2, 0){\textcolor{red}{$C_{2, R}$}};
\node[font=\tiny] at (2.2, 0.25){\textcolor{red}{$\Acal''_R$}};
\node at (5.5, 1) {$\Lambda_1(\delta_0)$};
\node at (5.5, -1) {$\Lambda_2(\delta_0)$};
\draw[draw=black, thick] (-5,1) .. controls (-.5,0) and (.5,0) .. (5,1);
\draw (0,0) circle (1cm);
\node at (0, 1.2) {$B_1$};
\end{tikzpicture}
 \caption{The annulus $\Acal''_R$ and  the disk $C_{2, R}$ after the surgery}
  \label{fig:circle3}
\end{figure}
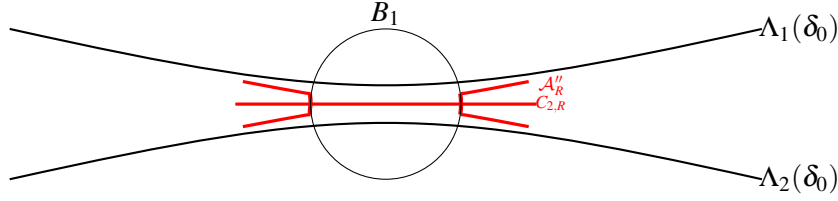

We now apply Proposition \ref{Equivariant version of the minimzing problem} again to $CHM_{R,1}$ in $\Omega \cap B_R$ to obtain a new surface $\Sigma_{R,1}$. We start the iteration by replacing $\Sigma_{R,0}$ by $\Sigma_{R,1}$ in \eqref{eq:Sigma0}. We can keep this process as long as the surface obtained is not connected. Each step of the iteration will make the weighted area decrease by a definite amount $\frac{1}{2}\mathcal{H}_w^2(P_{xy} \cap B_1) = 2\pi(e^{\frac{1}{4}}-1)$. On the other hand, the resulting surface after each iteration has weighted area bounded below by the $E$-minimizing current  $T$ in $\Omega \cap B_R$ with $\partial T =R\mathcal{L}[\Ccal[g]]$. By Section \ref{Preliminaries}, $T$ is given by a multiplicity 1 smoothly embedded self-expander, and has positive weighted area. So the iteration has to stop in finitely many steps. Thus at some step $j$, we obtain a connected self-expander $\Sigma_{R,j}$ with boundary $R\mathcal{L}[\Ccal]$ and genus bounded by $g$. 

Applying Riemann-Hurwitz formula as in \cite[Lemma B.1]{carlotto2020free}, we have
\begin{equation} \label{Riemann-Hurwitz formaula for Sigma_R}
\begin{split}
     2 - 2 g(\Sigma_{R,j}) - 3 & = \chi(\Sigma_R) \\
     & = (g + 1)\chi(\Sigma_{R,j} / \mathbf{C}_{g+1}) - kg \\
     & =(g + 1)(2 - 2g(\Sigma_{R,j} / \mathbf{C}_{g+1}) -3) - kg,
\end{split}   
\end{equation}
where $k$ is the number of intersection points of $\Sigma_{R,j}$ with $l_z$ including the origin. After simplifying equation \eqref{Riemann-Hurwitz formaula for Sigma_R}, we obtain
\begin{equation} \label{Riemann-Hurwitze formula 2}
2g(\Sigma_{R,j}) + 1 = (g + 1)(1 + 2g(\Sigma_R /\mathbf{C}_{g+1})) + kg.
\end{equation}
Since $g(\Sigma_{R,j}) \leq g$, the only integer solution for \eqref{Riemann-Hurwitze formula 2} is 
\[
g(\Sigma_{R,j}) = g,\, g(\Sigma_{R,j} /\mathbf{C}_{g+1}) = 0, \, k = 1.
\]
We then just set $\Sigma_R:=\Sigma_{R,j}$.
\end{proof}

\section{Complete Connected self-expanders with Positive genus and 3 Ends}
\noindent Now that we have constructed $\Grp[g]$-equivariant stable self-expanders $\Sigma_R \subset B_R$ with $\partial \Sigma_R = R \mathcal{L}(C)$ in each closed ball. We can now give the proof of Theorem \ref{the main result of the paper}

\begin{proof}[Proof of Theorem \ref{the main result of the paper}:]
    By Proposition \ref{Convergence to the limit self-expander} and \ref{prop:SigmaR}, for all $g > g_0$, $\Sigma_R$ converges to a complete connected $\Grp[g]$-equivariant self-expander $\Sigma_{\infty}$ smoothly of multiplicity 1. Moreover, $\Sigma_{\infty}$ is asymptotic to $\Ccal[g]$ and $g(\Sigma_{\infty}) \leq g$. Therefore, $\Sigma[g]$ has 3 ends. By choosing $R' > 0$ suitably large, $\Sigma_{\infty} \cap B_{R'}$ is properly embedded and has 3 boundary components. Applying Riemann-Hurwitz formula exactly the same as  \eqref{Riemann-Hurwitz formaula for Sigma_R}, we have $g(\Sigma_{\infty}\cap B_{R'}) = g$ and thus $g(\Sigma_{\infty}) = g$.
\end{proof}

\section{Behavior for high genus}
Recall that we define the rotationally symmetric triple cone $\Ccal_{\epsilon}$ in the introduction \eqref{eq:Cepsilon} and we have the following corollary.
\begin{corollary} \label{INfinitely many self-expanders}
For any $\epsilon > \delta_0$, there is a sequence of self-expanders $\{ \Sigma[g, \epsilon]\}_{g \in \mathbb{N}}$ such that the following hold.
\begin{enumerate}[(i)]
    \item $\Sigma[g,\epsilon]$ is asymptotic to $\Ccal_{\epsilon}$.
    \item $\Sigma[g, \epsilon]$ is connected and $g(\Sigma[g]) = g$.
    \item $\Sigma[g, \epsilon]$ has $\Grp[g]$-symmetry.
\end{enumerate}
\end{corollary}


Now we prove that by passing to a subsequence, the family $\{ \Sigma[g, \epsilon]\}_{g \in \mathbb{N}}$ will converge to the union of the hyperplane $P_{xy}$ and a self-expander annulus that is asymptotic to the double cone $\Ccal(\epsilon)$ as $g \to \infty$. For notational convenience, we may omit $\epsilon$ and just write $\Sigma[g]$ for $\Sigma[g, \epsilon]$ when $\epsilon$ is clear from the context.

\begin{proposition}\label{Behavior for high genus}
For any $\epsilon > \delta_0$, as $g\to\infty$, the sequence of self-expanders $\{ \Sigma[g, \epsilon]\}_{g \in \mathbb{N}}$ constructed in Corollary \ref{Infintely many expanders coming from the same cone} (by passing to a subsequence) converges to the union of the hyperplane and a self-expander annulus that is asymptotic to $\Ccal(\epsilon)$. The convergence is locally smooth away from the intersecting circle and is of multiplicity $1$.
\end{proposition}

\begin{proof}
Fix $\epsilon > \delta_0$, recall that $\Sigma[g]=\Sigma[g, \epsilon]$ has the quadratic volume growth by Proposition \ref{Convergence to the limit self-expander} and thus
\begin{equation} \label{quadratic area growth 1}
    \mathcal{H}^2(\Sigma[g] \cap B_R) \leq \frac{\mathcal{H}^1(\mathcal{L}[\Ccal_{\epsilon}])}{2} R^2
\end{equation}
for any $R > 0$. Allard's compactness theorem \cite{Allard1972Onthe} then implies the existence of a subsequence which converges to some stationary integral varifold $V=V_{\epsilon}$ with respect to the conformal metric $g_{Exp}$ as $g \to \infty$.

Fix any $R > 0$ and $x \in B_R$. In the following, we understand all geometric quantities with respect to the conformal metric $g_{Exp}$ in $\mathbb{R}^3$. Let $A_g$ denote the second fundamental form of the self-expander $\Sigma[g]$, the $\epsilon$-regularity theorem of Choi-Schoen \cite[Proposition 2]{choi1985space} states that there exists $\epsilon_0 > 0$ such that if $r \leq \epsilon_0$ and if
\begin{equation*}
    \int_{\Sigma[g] \cap B_r(x)} |A_g|^2 \, d\mu \leq \epsilon_0,
\end{equation*}
then there exists a constant $C$ which depends only on the background geometry and can hence be chosen uniformly with respect to the genus $g$ such that
\begin{equation*}
    \max_{0 \leq \sigma \leq r} \sigma^2 \sup_{B_{r-\sigma}(x)}|A_g|^2 \leq C.
\end{equation*}
Complementary to the set of points where this result applies, we consider as in \cite[(4.26)]{ketover2016free} (or see \cite[(3.2)]{Buzano2025Noncompact}):
\begin{equation*}
  \Upsilon^R : = \bigl\{x \in B_R : \inf_{r > 0}\left(\liminf_{g \to \infty} \int_{\Sigma[g] \cap B_r(x)} |A_g|^2 \, d\mu)  \right) \geq \epsilon_0 \bigr\}.
\end{equation*}
The $\Grp[g]$-equivariance of $\Sigma[g]$ implies that $\Upsilon^R$ is rotationally symmetric around the z-axis $l_z$. Clearly, we have $\Upsilon^{R_1} \subset \Upsilon^{R_2}$ whenever $0 < R_1 \leq R_2$. Also note that for any $x_0 \in B_R \backslash \Upsilon^R$, there exists $r > 0$ such that the second fundamental form $A_g$ of $\Sigma[g] \cap B_r(x_0)$ is bounded uniformly for all $g \in \mathbb{N}$. This then implies that the convergence of $\Sigma[g] \to V$ is smooth in $B_r(x_0)$. 

Now observe that $\Sigma[g]$ has the dihedral symmetry $\mathbf{D}_{g + 1}$ (recall \eqref{eq:dihedral}). Thus, the following claims follow verbatim from \cite[Claim 1, 2, 3, 4 of Theorem 1.3]{Buzano2025Noncompact} and hold for any $R > 0$ sufficiently large.
\begin{enumerate}[{\textbf{Claim }}1{:}]
    \item supp$||V||$, i.e., the support of $V$, contains the hyperplane $P_{xy}$.
    \item For any $x_0 \in \Upsilon^{R} \backslash l_z$, the genus of $\Sigma[g]$ restricted to any neighborhood of $x_0$ is unbounded as $g \to \infty$.
    \item $\Upsilon^R \subset P_{xy} \cup l_z$. Moreover, $(\Upsilon^R \cap P_{xy}) \backslash l_z$ consists of at most one circle around $l_z$ .
    \item The limit supp$||V|| \cap B_R$ is smooth and embedded away from $\Upsilon^R \cap P_{xy}$. And $\Upsilon^R \cap l_z$ is discrete.
\end{enumerate}
Now we need to show $\Upsilon^R \cap P_{xy} \backslash l_z$ is nonempty. For expanders, we do not have any analogues of the Frankel property and Brakke's regularity theorem \cite{Brakke1978Motion} (cf. \cite[p.54]{Colding2013Round}) of shrinkers. Thus we cannot apply the same arguments as in \cite[Claim 5 of Theorem 1.3]{Buzano2025Noncompact}.
\begin{lemma}
    If $R > 0$ is sufficiently large, $\Upsilon^R \cap P_{xy}$ contains a circle $\uC$ with positive radius.
\end{lemma}
\begin{proof}
Along with \textbf{Claim }3, for a contradiction, suppose that $\Upsilon^R \cap P_{xy} \subset \{0\}$ for all $R$. Then \textbf{Claim} 4 above implies that $\Upsilon^R \subset l_z$ is discrete for all $R$. Hence, by \textbf{Claim} 4 again, $\supp ||V|| \backslash \{0\}$ is a smooth, embedded self-expander. Being rotationally symmetric, each connected component of $\supp ||V|| \backslash P_{xy}$ that has $0$ in its closure has finite Euler characteristic. Hence, by the removable singularities theorem \cite[Proposition 1]{choi1985space}, each of these components extends smoothly to the origin, lies on one side of $P_{xy}$ (because $\supp ||V||$ is embedded away from $P_{xy}$) and touches $P_{xy}$ only in the origin. This is impossible by the strong maximum principle, and therefore no such components exist. As a result, $\supp ||V||$ is smoothly embedded and contains the hyperplane $P_{xy}$. 

Since $\Upsilon^R \subset l_z$ is discrete for all $R > 0$, the convergence of $\Sigma[g] \to V$ is smooth away from finitely many points in each closed ball $B_R$. Thus, $\supp ||V||$ is asymptotic to $\Ccal_{\epsilon}$, rotationally symmetric and invariant under the reflection $\UUU[c]$ for any $c \in \mathbb{R}$ (recall \ref{symmetry 2}). And hence, we have
\[
\supp ||V|| = P_{xy} \cup \Lambda_{\epsilon}^1 \cup \Lambda_{\epsilon}^2,
\]
where by Lemma \ref{graphical self-expander}, $\Lambda_{\epsilon}^1$ ($\Lambda_{\epsilon}^2$) is the unique expander asymptotic to the single cone $\Ccal(\epsilon) \cap \{z \geq 0\}$ ($\Ccal(\epsilon) \cap \{z \leq 0\}$) and strictly lies above (below) $P_{xy}$.

By Lemma \ref{graphical self-expander} again, $\Lambda_{\epsilon}^1$, $\Lambda_{\epsilon}^2$ are both graphical. By L'Hôpital's rule and coarea formula, we have
\begin{equation} \label{lower bound for local area}
    \begin{split}
    \lim_{R \to \infty} \frac{\mathcal{H}^2(\supp ||V|| \cap B_R)}{R^2} & = \lim_{R \to \infty} \frac{\int_{\supp ||V|| \cap \partial B_R} \frac{1}{|\nabla_{\supp ||V||}|x|| }\, d\mathcal{H}^1}{2R} \\
    & \geq \lim_{R \to \infty} \frac{\mathcal{H}^1(\supp ||V|| \cap \partial B_R)}{2R} \\
    & = \lim_{R \to \infty} \frac{\mathcal{H}^1(\Ccal_{\epsilon} \cap \partial B_R)}{2R} \cdot \frac{\mathcal{H}^1(\supp ||V|| \cap \partial B_R)}{\mathcal{H}^1(\Ccal_{\epsilon} \cap \partial B_R)} \\
    & = \frac{1}{2} \mathcal{H}^1(\Ccal_{\epsilon} \cap \partial B_1) = \frac{1}{2} \mathcal{H}^1(\mathcal{L}[\Ccal_{\epsilon}]),
\end{split}
\end{equation}
where the last equality follows from
\[
\lim_{R \to \infty} \frac{\mathcal{H}^1(\supp ||V|| \cap \partial B_R)}{\mathcal{H}^1(\Ccal_{\epsilon} \cap \partial B_R)} = 1,
\]
since $\supp ||V||$ is asymptotic to $\Ccal_{\epsilon}$. On the other hand, by \eqref{quadratic area growth 1}, $V$ also has the quadratic area growth: for any $R > 0$,
\begin{equation} \label{quadratic area growth 2}
    ||V||(B_R) \leq \frac{\mathcal{H}^1(\mathcal{L}[\Ccal_{\epsilon}])}{2} R^2.
\end{equation}

By Constancy theorem \cite[Theorem 41.1]{simon1984lectures}, $V = n_0 P_{xy} \cup n_1 \Lambda_{\epsilon}^1 \cup n_2 \Lambda_{\epsilon}^2$ where $n_0, n_1, n_2 \in \mathbb{N}$ are the multiplicities of each component. Combining with \eqref{quadratic area growth 2}, we have
\begin{equation} \label{quadratic area growth 4}
\begin{split}
  \lim_{R \to \infty} \frac{\mathcal{H}^2(\supp ||V|| \cap B_R)}{R^2} & =  \lim_{R \to \infty} \frac{\mathcal{H}^2(P_{xy} \cap B_R)}{R^2} + \frac{\mathcal{H}^2(\Lambda_{\epsilon}^1 \cap B_R)}{R^2} + \frac{\mathcal{H}^2(\Lambda_{\epsilon}^2 \cap B_R)}{R^2} \\
  & \leq \lim_{R \to \infty} n_0 \frac{\mathcal{H}^2(P_{xy} \cap B_R)}{R^2} + n_1 \frac{\mathcal{H}^2(\Lambda_{\epsilon}^1 \cap B_R)}{R^2} + n_2 \frac{\mathcal{H}^2(\Lambda_{\epsilon}^2 \cap B_R)}{R^2} \\
  & =  \lim_{R \to \infty} \frac{||V||(B_R)}{R^2} \leq \frac{\mathcal{H}^1(\mathcal{L}[\Ccal_{\epsilon}])}{2}.
\end{split}
\end{equation}
\eqref{lower bound for local area} then implies $n_0, n_1, n_2 = 1$ and
\begin{equation} \label{quadratic area growth 5}
    \lim_{R \to \infty} \frac{\mathcal{H}^2(\supp ||V|| \cap B_R)}{R^2} = \lim_{R \to \infty} \frac{||V||(B_R)}{R^2} = \frac{1}{2} \mathcal{H}^1(\mathcal{L}[\Ccal_{\epsilon}]).
\end{equation}
As a result, the convergence of $\Sigma[g]$ to $\supp ||V||$ is smooth away from finitely many points and of multiplicity $1$ in each closed ball $B_R$. By Allard's regularity theorem \cite{Allard1972Onthe}, the convergence is locally smooth everywhere. 

Now fix $R_0 > 1$. Due to the smooth convergence of multiplicity $1$, for all $g \in \mathbb{N}$ sufficiently large, $\Sigma[g] \cap B_{R_0}$ has the same topology as $\supp ||V|| \cap B_{R_0}$, which consists of three disjoint topological disks. Fix such $g \in \mathbb{N}$ sufficiently large, since $\Sigma[g]$ is asymptotic to $\Ccal_{\epsilon}$, there exists $R_g > R_0$ appropriately large such that $\Sigma[g] \cap B_{R_g}$ is properly embedded with boundary consists of $3$ simple closed curves, and $g(\Sigma[g] \cap B_{R_g}) = g$. Consider the surface $\Theta_g : = \Sigma[g] \cap (B_{R_g} \backslash B_{R_0})$. Then $\Theta_g$ has $6$ boundary components and $g(\Theta_g) = g$.  The quotient surface $\Theta_g / \mathbf{C}_{g + 1}$ is a compact Lipschitz surface with $6$ boundary components and no branch points as $\Theta_g \cap l_z = \emptyset$ (recall \eqref{eq:cyclic}). Now by the Riemann-Hurwitz formula as \cite[Lemma B.1]{carlotto2020free}, we have
\begin{equation}
\begin{split} \label{RH formula 1}
    2- 2g - 6 & = \chi(\Theta_g)\\
    & = (g+1) \chi(\Theta_g / \mathbf{C}_{g + 1}) \\
    & = (g+ 1) (2 - 2g(\Theta_g / \mathbf{C}_{g + 1}) - 6) \\
\end{split}
\end{equation}
After simplifying (\ref{RH formula 1}), we obtain
\begin{equation} \label{RH formula 2}
    (g+ 1) (g(\Theta_g / \mathbf{C}_{g + 1}) + 2) = g + 2.
\end{equation}
We arrive at a contradiction since there is no positive integer solution for \eqref{RH formula 2}.
\end{proof}
Now by \textbf{Claim} 3 above, $ \uC := \Upsilon^R \backslash l_z$ is exactly one circle centered at the origin. Now  \cite[Claim 6 of Theorem 1.3]{Buzano2025Noncompact} gives the following:
\begin{enumerate}[(i)]
    \item $\Upsilon^R = \uC$ for all $R > 0$ sufficiently large.
    \item The blow-up of $\supp ||V||$ around any point $x_0 \in \uC$ is given by two orthogonal planes.
    \item $\supp ||V|| = P_{xy} \cup \uS$ where $\uS$ is a smooth, embedded rotationally symmetric self-expander intersecting the plane $P_{xy}$ orthogonally in $\uC$.
\end{enumerate}
Thus, the profile $\uS \cap \{(x, 0, z) : x \geq 0\}$ is a smooth, connected curve intersecting the $x$-axis exactly once. In particular, it cannot be closed. Therefore, $\uS$ has genus zero. Since the convergence of $\Sigma[g] \to V$ is smooth away from $\uC$, $\supp ||V|| = P_{xy} \cup \uS$ is also asymptotic to $\Ccal_{\epsilon}$. As a result, $\uS$ must be a self-expander annulus that is asymptotic to the double cone $\Ccal(\epsilon)$.

Now it remains to show that the convergence is of multiplicity $1$. By Constancy theorem \cite[Theorem 41.1]{simon1984lectures}, $V = n_0 P_{xy} \cup n_1 \uS$ where $n_0, n_1 \in \mathbb{N}$ are the multiplicities of each component. Proceeding the same as \eqref{lower bound for local area}-\eqref{quadratic area growth 5}, we have $n_0, n_1 = 1$ and the proof is complete.
\end{proof}

\begin{remark}
The limit self-expander annulus that is asymptotic to the double cone $\Ccal(\epsilon)$ may depend on the subsequence we choose, i.e., for different subsequences of $\{ \Sigma[g, \epsilon]\}_{g \in \mathbb{N}}$, the self-expander annuli contained in their limits may vary. This is due to the non-uniqueness of the self-expander annulus that is asymptotic to the double cone $\Ccal(\epsilon)$.
\end{remark}

\appendix
\section{Proof of Proposition \ref{Convergence to the limit self-expander} and \ref{Equivariant version of the minimzing problem}}
\renewcommand{\thetheorem}{\Alph{section}.\arabic{theorem}}
We include the proof for Proposition \ref{Convergence to the limit self-expander} and \ref{Equivariant version of the minimzing problem} in this section.

\begin{proposition}[Proposition \ref{Convergence to the limit self-expander}]\label{appendix theorem for the convergence to the limit complete self-expander}
Let $\Ccal \subset \mathbb{R}^3$ be a smooth embedded cone. Suppose that for all sufficiently $R > 0$, there exists a properly embedded self-expander $\Sigma_R \subset B_R$ satisfying the following.
\begin{enumerate}[(i)]
    \item $\Sigma_R$ is connected.
    \item $g(\Sigma_R)$ is uniformly bounded above by $C$. 
    \item $\partial \Sigma_R = R \mathcal{L}[\Ccal]$.
\end{enumerate}
Then there is a complete embedded self-expander $\Sigma_{\infty} \subset \mathbb{R}^3$ such that after passing to subsequences, $\Sigma_R$ converges to $\Sigma_{\infty}$ smoothly with multiplicity 1, and $\Sigma_{\infty}$ is asymptotic to $\Ccal$. Moreover, the following hold.
\begin{enumerate}[(i)]
    \item $\Sigma_{\infty}$ is connected.
    \item $g(\Sigma_{\infty}) \leq C$.
    \item $\Sigma_{\infty}$ has the quadratic volume growth: $\mathcal{H}^2(\Sigma_{\infty} \cap B_R) \leq \frac{\mathcal{H}^1(\mathcal{L}[\Ccal])}{2} R^2$ for any $R > 0$.
\end{enumerate}
\end{proposition}
\begin{remark}
The proof of Proposition \ref{Convergence to the limit self-expander} above is a slight modification of  the proofs of \cite[Lemma 8.2]{Bernstein2023Integer} and \cite[Theorem 6.3]{ding2020minimal}. We adapt the same barriers constructed by \cite[Lemma 8.2]{Bernstein2023Integer}. The assumption on the genus upper bound is merely to guarantee the compactness of the sequence since we do not have the area-minimizing property as in \cite[Lemma 8.2]{Bernstein2023Integer} and \cite[Theorem 6.3]{ding2020minimal}. Once we establish the existence of the limit self-expander $\Sigma_{\infty}$, the multiplicity 1 smooth convergence follows verbatim as in \cite[Theorem 6.3]{ding2020minimal}, and that $\Sigma_{\infty}$ is asymptotic to $\Ccal$ follows the same as in \cite[Lemma 8.2]{Bernstein2023Integer}.
\end{remark}
\begin{proof}
The proof is proceeded by a barrier argument. Let us first give the construction of barriers.

Consider the following family of functions depending on parameters $\mathbf{v} \in \mathbb{S}^2, \eta > 0, h \geq 0$:
\[
f_{\mathbf{v}, \eta, h}(\mathbf{x}) := 4 + |\mathbf{x}|^2 - (1+ \eta^{-2})(\mathbf{x} \cdot \mathbf{v})^2 + h.
\]

We observe that the connected closed set:
\[
E_{\mathbf{v}, \eta, h} := \{ \mathbf{x} \in \mathbb{R}^3 : f_{\mathbf{v}, \eta, h}(\mathbf{x}) \leq 0 \text{ and } \mathbf{x} \cdot \mathbf{v} \geq 0\} 
\]
has boundary asymptotic to the connected, rotationally symmetric cone:
\[
\Ccal_{\mathbf{v}, \eta} := \{|\mathbf{x}|^2 = (1 + \eta^{-2})(\mathbf{x} \cdot \mathbf{v})^2, \mathbf{x} \cdot \mathbf{v} \geq 0 \}
\]
which lies in the half space $\{\mathbf{x} \cdot \mathbf{v} \geq 0\}$ and has axis parallel to $\mathbf{v}$ and cone angle $\eta^{-1}$. Furthermore, $E_{\mathbf{v}, \eta, h}$ lies entirely in the component of $\mathbb{R}^3 \backslash \Ccal_{\mathbf{v}, \eta}$ that contains $\mathbf{v}$. By construction, $f_{\mathbf{v}, \eta, h} > 0$ on $\{|\mathbf{x} \cdot \mathbf{v}| < \eta \sqrt{4 + h} \, \}$ and so
\begin{equation} \label{E_p lies above the plane vertical to v}
    E_{\mathbf{v}, \eta, h} \cap \{\mathbf{x} \cdot \mathbf{v} < \eta \sqrt{4 + h} \, \} = \emptyset.
\end{equation}
A straightforward computation gives that on any self-expander, say $\Sigma$:
\begin{equation} \label{differential inequality satisfied by f}
    \Delta_{\Sigma} f_{\mathbf{v}, \eta, h} + \frac{1}{2} \mathbf{x} \cdot \nabla_{\Sigma}f_{\mathbf{v}, \eta, h} - f_{\mathbf{v}, \eta, h} = -2(\eta^{-2} + 1)|\nabla_{\Sigma}(\mathbf{x} \cdot \mathbf{v})|^2 - h \leq 0.
\end{equation}

Now it is ready to check that for each $p \in \mathcal{L}[\Ccal]$ and choice of normal $N(p)$ at $p$, there are $\eta(p) >0, h(p) >0, \mathbf{v}_{\pm}(p) \in \mathbb{S}^2$ so that $E_p := E_{\mathbf{v}_+(p), \eta(p), h(p)} \cup E_{\mathbf{v}_-(p), \eta(p), h(p)}$ satisfies $E_p \cap \Ccal = \emptyset$, $\mathcal{L}[\Ccal_{\mathbf{v}_{\pm}, \eta}]$ lies on the $\pm N$ side of $\mathcal{L}[\Ccal]$ and touches $\mathcal{L}[\Ccal]$ only at $p$.  Now we claim that $\Sigma_R \cap E_p = \emptyset$ for any $p \in \mathcal{L}[\Ccal]$.

Indeed, fix $p\in  \mathcal{L}[\Ccal]$. Since $E_{p} \cap \Ccal = \emptyset$ , for any $p^{\prime} \in \partial \Sigma_R \cap \{\mathbf{x} \cdot \mathbf{v}_+ \geq 0 \}$, we have
\begin{equation}
    f_{\mathbf{v}_+, \eta, h}(p^{\prime}) > 0,
\end{equation}
where we use notation $v_{\pm} := v_{\pm}(p), \eta := \eta(p), h:= h(p)$. Clearly, $f_{\mathbf{v}_{+}, \eta, h} > 0$ holds on $\{\mathbf{x} \cdot \mathbf{v}_{+} = 0 \}$. Since $\partial (\Sigma_R \cap \{\mathbf{x} \cdot \mathbf{v}_{ +} \geq 0 \}) \subset \{\mathbf{x} \cdot \mathbf{v}_{+} = 0 \} \cup (\partial \Sigma_R \cap \{\mathbf{x} \cdot \mathbf{v}_{+} \geq 0 \})$, by \ref{differential inequality satisfied by f}, the strong maximum principle implies that $f_{\mathbf{v}_{+}, \eta, h} > 0$ on $\Sigma_R \cap \{\mathbf{x} \cdot \mathbf{v}_{+} \geq 0 \}$. Thus, $E_{\mathbf{v}_+, \eta, h} \cap \Sigma_R = E_{\mathbf{v}_+, \eta, h} \cap (\Sigma_R \cap \{\mathbf{x} \cdot \mathbf{v}_{+} \geq 0 \}) = \emptyset$ since $E_{\mathbf{v}_+, \eta, h} \subset \{\mathbf{x} \cdot \mathbf{v}_{+} \geq 0 \}$. Replacing $\mathbf{v}_+$ by $\mathbf{v}_-$ and proceeding the same as above, we can show that $E_p \cap \Sigma_R = \emptyset$ for all $p \in \mathcal{L}[\Ccal]$ and the claim is shown.
\begin{figure}[H]
\centering
\begin{tikzpicture}
\filldraw[color=black!60, fill=black!5, very thick] (0.436, 4.981) .. controls (0.02588, 0.09659)  .. (2.113, 4.5315);
\draw[color=black!60, fill=black!5, very thick] (-0.436, 4.981) .. controls (-0.02588, 0.09659)  .. (-2.113, 4.5315);
\draw (0, 0) -- (2.5, 4.443);
\draw (0, 4)[draw=red, thick] arc[start angle=90, end angle=85, radius=4];
\node[font=\tiny] at(0.17448, 4.2){\textcolor{red}{$\gamma_r$}};
\node at (3.1, 4){$\Ccal_{\mathbf{v}_+, \eta}$};
\node at (-3, 4){$\Ccal_{\mathbf{v}_-, \eta}$};
\draw (0, 0) -- (-2.5, 4.43);
\filldraw[black] (0,1) circle (1pt) node[anchor=east]{$p$};
\filldraw[red] (0,4.5) circle (1pt) node[anchor=south]{$Rp$};
\node at (0.85, 3.3){$E_{\mathbf{v}_+, \eta, h}$};
\node at (-0.85, 3.3){$E_{\mathbf{v}_-, \eta, h}$};
\node at (0.6, 0.2){\textcolor{red}{$\Sigma_R$}};
\node at (0, 5.2){$\Ccal$};
\draw[draw=red, very thick] (0, 4.5) .. controls (0.1, 0) .. (.5, 0);
\draw (0, 0) -- (0, 5);
\end{tikzpicture}
 \caption{$\Sigma_R$ being disjoint from $E_p$}
  \label{fig:Barriers}
\end{figure}
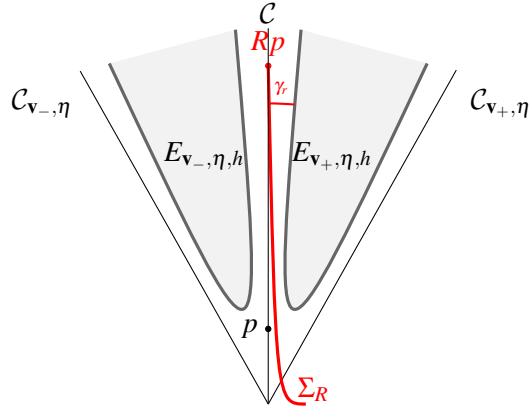
Since $\mathcal{L}[\Ccal]$ is compact, we can find:
\[
 \eta_{max} := \max_{p \in \mathcal{L}[\Ccal]}\eta (p) > 0, \,h_{max} := \max_{p \in \mathcal{L}[\Ccal]} h(p) > 0.
\]
Notice that $E_{\mathbf{v}, \eta_1, h_1} \subset E_{\mathbf{v}, \eta_2, h_2}$ holds whenever $\eta_1 \geq \eta_2, h_1 \geq h_2$ holds. Thus, we redefine $E_p := E_{\mathbf{v}_+(p), \eta_{max}, h_{max}} \cup E_{\mathbf{v}_-(p), \eta_{max}, h_{max}}$ and $E_p \cap \Sigma_R = \emptyset$ still holds for any $p \in \mathcal{L}[\Ccal]$.

Let $\{ \mathcal{L}_1, \dots, \mathcal{L}_m \}$ be the collection of all the components of $\mathcal{L}[\Ccal]$. Fix $i \in \{1, \dots, m\}$, we define $D_i := \bigcup_{p \in \mathcal{L}_i} E_p$. $D_i$ consists of 2 domains $D_i^{\pm}$, which lies on the $\pm N$ side of $\mathcal{L}_i$. Take $R_0 = \eta_{max} \sqrt{4 + h_{max}} + 1$, observe that  $\partial D_i^{+} \backslash B_{R_0}$ has 2 components and we denote by $T_i^+$ the component that has closer Hausdorff distance to $\{r \mathcal{L}_i : r > 0 \}$. We do the same for $\partial D_i^{-} \backslash B_{R_0}$ and pick $T_i^-$ to be the component  closer to $\{r \mathcal{L}_i : r > 0 \}$. Then we have:
\begin{equation}
    T_i^{\pm} = \bigcup_{r \geq R_0} \partial N_{l(r)}(r \mathcal{L}_i)
\end{equation}
where $N_{l_i(r)}(r \mathcal{L}_i)$ is the $l(r)$-tubular neighborhood of $r\mathcal{L}_i$ on the sphere $\partial B_r$ and $l(r)$ is the distance between $\partial B_r \cap T_i^+$ and $r\mathcal{L}_i$ on $\partial B_r$. Now we give the explicit expression of $l(r)$. Let $p \in \mathcal{L}_i$, $P$ be the plane spanned by $\mathbf{x}(p), \mathbf{v}_+(p)$, and $q$ be the intersection point of $P$ and $T_i^+$. Then $l(r)$ is just the length of the geodesic segment $\gamma_r \subset P \cap \partial B_r$ between $rp$ and $q$. By direct calculation, we have
\begin{equation}
    l(r) = (\arctan(\frac{1}{\eta_{max}}) - \arctan \left(\sqrt{\frac{r^2(1 + \eta_{max}^{-2})}{4 + h_{max} + r^2} -1} \right))r = O(\frac{1}{r}) \text{ as }r \to \infty,
\end{equation}
and hence
\begin{equation}
    \lim_{r \to \infty} l(r) = 0.
\end{equation}
Consequently, we have
\begin{equation} \label{Hausdorff convergence of sigma R}
    d_{H}(\Sigma_R \cap \partial B_r, r\mathcal{L}[\Ccal]) \leq 2 l(r) \text{ for all }R_0 < r < R,
\end{equation}
where $d_H(A, B) : = \max \{\max_{x \in A} d(x, B), \max_{x \in B} d(x, A) \}$ is the Hausdorff distance between any two sets $A, B \subset \mathbb{R}^3$ and $d$ is the distance on $\mathbb{R}^3$ with the standard Euclidean metric.
Now we have finished the construction of barriers $\bigsqcup_{i = 1}^m T_i^{\pm}$. We will show below that  $\{\Sigma_R\}_{R > 0}$ will limit to a self-expander $\Sigma_{\infty}$ and $\Sigma_{\infty}$ is asymptotic to $\Ccal$. 

By construction, $\Sigma_R \cap \bigcup_{p \in \mathcal{L}[\Ccal]} E_p = \emptyset$, and hence for any $R > \tilde{R} > R_0$, $\Sigma_R \backslash B_{\tilde{R}} \subset \bigcup_{r > \tilde{R}}(\bigsqcup_{i = 1}^m N(l_i(r))(r \mathcal{L}_i))$ has at least $m$ components. We claim that $\Sigma_R \backslash B_{\tilde{R}}$ has exactly $m$ components. Indeed, consider the function $u(\mathbf{x}) : = |\mathbf{x}|^2 + 4$. On any self-expander $\Sigma$, we have
\begin{equation} \label{inequality satisfied by u}
    \Delta_{\Sigma} u + \frac{1}{2} \mathbf{x} \cdot \nabla_{\Sigma}u - u = 0.
\end{equation}
Suppose $\Sigma_R \backslash B_{\tilde{R}}$ has more than $m$ components. Then there must exist one component $\Xi_R \subset \Sigma_R \backslash B_{\tilde{R}}$ such that $\partial \Xi_R \subset \partial B_{\tilde{R}}$. Thus, $u$ achieves its maximum in the interior of $\Xi_R$, contradicting the strong maximum principle. We have then arrived at the claim. As a result, since $\Sigma_R$ is connected, $\Sigma_R \cap B_{R^{\prime}}$ must be nonempty and connected for all $R > \tilde{R} > R_0$.

Now we give the local area bound on $\Sigma_R$. Equation \ref{eq:Laplacian} implies that $\Delta_{\Sigma_R} |X|^2 = 4 + 4H_{\Sigma_R}^2 \geq 4$, and hence
\begin{equation} \label{eq:local area bound}
    4 \mathcal{H}^2(\Sigma_R) \leq \int_{\Sigma_R} \Delta_{\Sigma_R} |X|^2 \,  \leq 2 \int_{\partial \Sigma_R} |X| \leq 2 R  \mathcal{H}^1(\partial \Sigma_R) = 2R^2 \mathcal{H}^1(\mathcal{L}[\Ccal]) .
\end{equation}
Applying the monotonicity formula \ref{Monotonicity Formula}, \eqref{eq:local area bound} implies that
\begin{equation} \label{Density upper bound from the Monotonicity formula}
\begin{split}
    \frac{\mathcal{H}^2(\Sigma_R \cap B_r)}{r^2} & \leq \frac{\mathcal{H}^2(\Sigma_R \cap B_R)}{R^2} \\
    & = \frac{\mathcal{H}^2(\Sigma_R)}{R^2} \\
    & \leq \frac{1}{2} \mathcal{H}^1(\mathcal{L}[\Ccal])
\end{split}
\end{equation}
for all $0 < r  < R$. Thus,  for any $R > \tilde{R} > R_0$, we have 
\begin{equation}
    \mathcal{H}^2_w(\Sigma_R \cap B_{\tilde{R}}) \leq  e^{\frac{\tilde{R}^2}{4}} \mathcal{H}^2(\Sigma_R \cap B_{\tilde{R}}) \leq \frac{1}{2} \tilde{R}^2 e^{\frac{\tilde{R}^2}{4}} \mathcal{H}^1(\mathcal{L}[\Ccal]) = C_1(\tilde{R}, \Ccal)
\end{equation}
where $C_1(\tilde{R}, \Ccal)$ is a constant only depending on $\tilde{R}$  and the cone $\Ccal$. Thus, $\Sigma_R$ has uniformly  bounded genus and weighted area in the closed ball $B_{\tilde{R}}$. By Ilmanen's \cite{ilmanen1998lectures} localized Gauss-Bonnet argument, we have
\[
\sup_{R > \tilde{R}}\int_{\Sigma_R \cap B_{\frac{\tilde{R}}{2}}} |A|^2   \leq C_2(\tilde{R}, g, C),
\]
where $A$ is the second fundamental form of $\Sigma_R \cap B_{\tilde{R}}$ in $B_{\tilde{R}}$ with respect to the conformal metric $g_{Exp}$ \eqref{eq:gexp}. By Choi and Schoen's curvature estimates \cite{choi1985space}, there is an embedded self-expander $\Sigma_{\infty} \subset \mathbb{R}^3$ such that after passing to subsequences, in each closed ball $B_{\frac{\tilde{R}}{2}}$, $\{\Sigma_R\}_{R > \tilde{R}}$ converges to $\Sigma_{\infty}$ (possibly with multiplicity) smoothly away from finitely many points. Since $\Sigma_R \cap B_{R^{\prime}}$ is connected for any $R > R^{\prime} > R_0$, $\Sigma_{\infty} \cap B_{R^{\prime}}$ must be connected for any $R^{\prime} > R_0$. Thus, $\Sigma_{\infty}$ is connected. By Constancy theorem \cite[Theorem 41.1]{simon1984lectures}, the convergence of $\{\Sigma_R\}_{R > \tilde{R}}$  to $\Sigma_{\infty}$ is of multiplicity $k$ for some integer $k \in \mathbb{Z}^+$. Thus, for any $r > 0$, we have
\[
\lim_{R \to \infty}d_H(\Sigma_R \cap \partial B_r, \Sigma_{\infty} \cap \partial B_r) = 0.
\]
Combing with \eqref{Hausdorff convergence of sigma R}, for any $R > r > R_0$, we obtain:
\begin{equation} \label{eq:Asymptotic conical behavior}
\begin{split}
    d_H(r  \mathcal{L}[\Ccal], \Sigma_{\infty} \cap \partial B_r) & \leq d_{H}(\Sigma_R \cap \partial B_r, r\mathcal{L}[\Ccal]) + d_H(\Sigma_R \cap \partial B_r,\Sigma_{\infty} \cap \partial B_r) \\
    & \leq 2l(r) + d_H(\Sigma_R \cap \partial B_r,\Sigma_{\infty} \cap \partial B_r).
\end{split}
\end{equation}
Let $R \to \infty$ in \eqref{eq:Asymptotic conical behavior} above, we then have:
\begin{equation} \label{Hausdorff convergence of Sigma infinity}
    d_H(r  \mathcal{L}[\Ccal], \Sigma_{\infty} \cap \partial B_r) \leq 2l(r) \to 0 \text{ as }r \to \infty.
\end{equation}
Combining (\ref{Hausdorff convergence of sigma R}),  (\ref{Hausdorff convergence of Sigma infinity}), one obtains that $\Sigma_{\infty}$ is $C^0$ asymptotic to $\Ccal$. Hence, we have
\[
\lim_{r \to \infty} \frac{\mathcal{H}^1(\Sigma_{\infty} \cap \partial B_r)}{\mathcal{H}^1(C \cap \partial B_r)} = 1
\]
By L'Hôpital's rule and the coarea formula, we have
\begin{equation} \label{Density lower bound}
\begin{split}
    \lim_{r \to \infty} \frac{\mathcal{H}^2(\Sigma_{\infty} \cap B_r)}{r^2} & = \lim_{r \to \infty} \frac{\int_{\Sigma_{\infty} \cap \partial B_r} \frac{1}{|\nabla_{\Sigma_{\infty}}|x|| }\, d\mathcal{H}^1}{2r} \\
    & \geq \lim_{r \to \infty} \frac{\mathcal{H}^1(\Sigma_{\infty} \cap \partial B_r)}{2r} \\
    & = \lim_{r \to \infty} \frac{\mathcal{H}^1(\Ccal \cap \partial B_r)}{2r} \cdot \frac{\mathcal{H}^1(\Sigma_{\infty} \cap \partial B_r)}{\mathcal{H}^1(\Ccal \cap \partial B_r)} \\
    & = \frac{1}{2} \mathcal{H}^1(\Ccal \cap \partial B_1) \cdot \lim_{r \to \infty} \frac{\mathcal{H}^1(\Sigma_{\infty} \cap \partial B_r)}{\mathcal{H}^1(\Ccal \cap \partial B_r)}\\
    & = \frac{1}{2} \mathcal{H}^1(\mathcal{L}[\Ccal]).
\end{split}
\end{equation}
On the other hand, by (\ref{Density upper bound from the Monotonicity formula}), we have
\begin{equation} \label{local area bound for expanders}
    k \frac{\mathcal{H}^2(\Sigma_{\infty}\cap B_{r})}{r^2} = \lim_{R \to \infty} \frac{\mathcal{H}^2(\Sigma_R \cap B_r)}{r^2}
    \leq \frac{1}{2} \mathcal{H}^1(\mathcal{L}[\Ccal])  \text{ for all } r > 0,
\end{equation}
which implies that  $k \lim_{r \to \infty} \frac{\mathcal{H}^2(\Sigma_{\infty} \cap B_r)}{r^2} \leq \frac{1}{2} \mathcal{H}^1(\mathcal{L}[\Ccal])$. Therefore, $k = 1$ and the convergence of $\Sigma_R$ to $\Sigma_{\infty}$ is smooth everywhere in each closed ball by Allard's regularity theorem \cite{Allard1972Onthe}. As genus cannot increase in the smooth convergence limit, the genus bound of $\Sigma_{\infty}$ follows. Plugging $k = 1$ in (\ref{local area bound for expanders}), the quadratic volume growth of $\Sigma_{\infty}$ then follows.
\end{proof}

 Now we prove Proposition \ref{Equivariant version of the minimzing problem}. Consider the closed ball $B_R$ in $\mathbb{R}^3$ equipped with the conformal metric $g_{Exp}$. $B_R$ is strictly convex under the conformal metric and has constant mean curvature given by $H \equiv 2e^{\frac{-R^2}{8}} (\frac{1}{R}+ \frac{R}{4})$. For notational convenience, from now on we abbreviate the finite symmetry group $\Grp[g]$ by $\Grp$  and the singular locus $S[g] = S_1[g] \cup S_2[g] \cup S_3[g]$ by $S = S_0 \cup S_1 \cup S_2$ (recall \eqref{eq:G} and (\ref{lem:S})). 

\begin{definition}{\cite[Definition 3.7]{ketover2016equivariant}} \label{G-equivariant}
 Let $G$ be any finite group acting by isometries on $(\mathbb{R}^3, g_{Exp})$. An isotopy $\Phi(t)$ is called \emph{$G$-equivariant} if $\Phi(t) = g^{-1} \circ \Phi(t) \circ g$ for all $t$ and $g \in G$. Likewise, a vector field $X$ is called \emph{$G$-equivariant} if $g^{*}(X) = X$ for all $g \in G$. 
 
 We say a varifold $V$ is \emph{$G$-stationary} in an open set $U \subset (\mathbb{R}^3, g_{Exp})$ if for any $G$-equivariant vector field $X$ compactly supported on $U$, the first variation $\delta V (X) = 0$. Any $G$-stationary varifold $V$ is a stationary varifold \cite[Lemma 3.8]{ketover2016equivariant}. 
 
 For a $G$-equivariant minimal surface $\Sigma$ in $U$, we say $\Sigma$ is \emph{$G$-stable} if the second variation of $\Sigma$ is nonnegative with respect to all $G$-equivariant isotopies in $U$. 
\end{definition}

\begin{notation}\label{G-isotopies}
    Let $G$ be any finite group acting by isometries on $(\mathbb{R}^3, g_{Exp})$. For open sets $U,W$ with $U \subset W \subset B_R$, we denote by $\mathbf{Js}(U, W)$ ($\mathbf{Js}_G(U, W)$) the set consisting of all smooth ($G$-equivariant) isotopies of $W$ which leave $W \backslash U$ fixed, i.e.,
    \begin{equation*}
        \mathbf{Js}(U, W)=\{\varphi\in C^{\infty}([0,1]\times W, W): \varphi|_{W\setminus U}\equiv \mathrm{id}\},
    \end{equation*}
    (and for $G$-equivariant isotopies respectively). We may omit $W$ and just write $\mathbf{Js}(U)$ ($\mathbf{Js}_{G}(U)$) when $W$ is $B_R$. Under this notation, $\mathbf{Js}(B_R)$ ($\mathbf{Js}_{G}(B_R)$) is just the set of all smooth ($G$-equivariant) isotopies of $B_R$.
\end{notation}

\begin{definition}\label{def:Gmini}
    Let $\Sigma \subset W$ be any embedded surface. If $\{\varphi^k \} \subset \mathbf{Js}(U, W)$ and
    \[
    \lim_{k \to \infty} \mathcal{H}^2_w(\varphi^k(1, \Sigma)) = \inf_{\varphi \in \mathbf{Js}(U, W)} \mathcal{H}^2_w(\varphi(1, \Sigma)),
    \]
    then we say $\{\varphi^k(1, \Sigma)\}^k$ is a \emph{minimizing sequence for Problem $(\Sigma, \mathbf{Js}(U, W))$}. If $\Sigma \subset W$ is a $G$-equivariant embedded surface and $\{\varphi^k_{G} \} \subset \mathbf{Js}_{G}(U, W)$ and
    \[
     \lim_{k \to \infty} \mathcal{H}^2_w(\varphi^k_{G}(1, \Sigma)) = \inf_{\varphi \in \mathbf{Js}_{G}(U, W)} \mathcal{H}^2_w(\varphi(1, \Sigma)),
    \]
we say $\{\varphi_{G}^k(1, \Sigma)\}^k$ is  a \emph{$G$-equivariant minimizing sequence for Problem $(\Sigma, \mathbf{Js}_{G}(U, W))$}.
\end{definition}

\begin{notation}
    For any $\mathbf{x} \in \mathbb{R}^3$, we define 
    \begin{equation*}
        \mathcal{AN}_r(\mathbf{x}) : = \{ \text{open annulus } B_t(\mathbf{x}) \backslash \overline{B_{\tau}(x)} \text{ where }0 < \tau < t < r \}.
    \end{equation*}
\end{notation}

Now we give the proof of Proposition \ref{Equivariant version of the minimzing problem}. For illustrational convenience, we consider the ambient space to be $B_R$, the proof can be easily generalized to the case where the ambient space has piecewise smooth boundaries and nonnegative mean curvature.
\begin{proposition}[Proposition \ref{Equivariant version of the minimzing problem}]\label{prop:2}
    Let $\Sigma \subset B_R$ be a properly embedded $\Grp$-equivariant  surface with smooth boundary $\partial  \Sigma$ intersecting $S \cap \partial B_R$ transversally. Suppose that $\{\Delta^j := \varphi^j(1, \Sigma)\}_j$ is a $\Grp$-equivariant minimizing sequence for Problem $(\Sigma, \mathbf{Js}_{\Grp}(B_R))$ converging to a stationary varifold $V$. Then $V = \bigcup_{k = 1}^N \Gamma_k$ is a union of  connected disjoint embedded minimal surfaces $\Gamma_1, \dots, \Gamma_N \subset B_R$ with respect to the metric $g_{Exp}$, each with multiplicity 1. Moreover, the following hold.
    \begin{enumerate}[(i)]
     \item $\partial \Gamma_1 \cup \dots \cup \partial\Gamma_k = \partial \Sigma$.
     \item $\sum\limits_{k = 1}^N g(\Gamma_k) \leq g(\Sigma)$.
    \end{enumerate}
\end{proposition}
\begin{remark}
To show the regularity of $V$ up to the boundary, we need to distinguish the following 4 cases.
\begin{enumerate}[(i)]
    \item The interior regularity on $B_R$ away from $S$ follows by Meeks-Simon-Yau \cite[Theorem 1]{meeks1982embedded} and the theory of replacements \cite[Proposition 6.3]{colding2003min}. 
    \item The boundary regularity on $\partial B_R$ away from $S$ follows by \cite[Lemma 8.1]{de2010genus}.
    \item The regularity near points on $S_1 \backslash (\partial B_R \cup l_z)$ (recall \ref{lem:S}) follows from \cite[Proposition 7.3]{ketover2016free}.
    \item The regularity near points on $S_2 \backslash \partial B_R$ (recall \ref{lem:S}) can be treated as the problem near points at a free boundary and follows from \cite[Theorem 7.3, Theorem 8.3]{Li2015General}.
\end{enumerate}
It then remains to show the regularity on $\mathcal{W}:= (\text{supp}||V|| \cap S) \cap (l_z \cup \partial B_R)$. Due to the symmetry and the regularity of $V$ shown above, a careful examination will show that $\mathcal{W}$ is a finite set. By \cite[Proposition 1]{choi1985space}, these points are removable singularities.
\end{remark}

\begin{proof}
Since $\{\Delta_j\}_j$ is a $\Grp$-equivariant minimizing sequence, $V$ is a $\Grp$-stationary varifold and hence  stationary as discussed in Definition \ref{G-equivariant}. It remains to show the regularity of $V$.

First note that by \cite[Lemma 3.4 (2)]{ketover2016equivariant} and the $\Grp$-symmetry, $\Sigma$ must contain all the horizontal axes $\bigcup_{j = 1}^{g + 1} l_{\frac{(2j-1)\pi}{2(g+1)}}$, intersects $l_z$ orthogonally at finitely many points, and orthogonal to the vertical planes $P_{\frac{j \pi}{g+1}}$ for all $j \in \{0, \dots, g\}$. By \cite[Lemma 3.6]{ketover2016equivariant}, each $\Delta_j$ in the $\Grp$-equivariant minimizing sequence are also orthogonal to $l_z$ and has the same number of intersection points with $l_z$ as $\Sigma$. Thus, after passing to a subsequence, we let $\mathcal{P}$  be the finite set of limit points of these intersection points. Since $\{ \Delta_j \}_j$ is a sequence of surfaces of fixed genus, by \cite[Lemma 1.0.14]{Colding2004Space}  or \cite[Proposition 4.12 (5)]{ketover2016equivariant}, there exists a finite set of points $\mathcal{Q}$ such that for any $x \in B_R \backslash \mathcal{Q}$, $g(\Sigma_j \cap B_r(x)) = 0$ for $j \in \mathbb{N}$ sufficiently large and $r > 0$ small enough ,while for any $x \in \mathcal{Q}$, there exists $R_x, g_x > 0$ and a sequence $\{R_{x, j}\}_j \to 0$ such that $g(B_{R_x}(x) \cap \Sigma_j) = g(B_{R_{x, j}}(x) \cap \Sigma_j) = g_x$ for all $j \in \mathbb{N}$ sufficiently large (the proof of \cite[Lemma 1.0.14]{Colding2004Space} works for any sequence of surfaces with fixed genus, not necessarily minimal). Thus, by shrinking the annuli appropriately small, we can define a $\Grp$-equivariant function $r: B_R \to \mathbb{R}^+$ such that for all $j \in \mathbb{N}$ sufficiently large, we have $g(\Sigma_j \cap An) = 0$ and $\Sigma_j \cap An \cap l_z = \emptyset$ for every $An \in \mathcal{AN}_{r(x)}(x)$. Due to technical purposes, we first show the interior regularity of $V$ in $B_R \backslash l_z$. By \cite[Proposition 6.3]{colding2003min}, it suffices to show $V$ has good replacements (see \cite[Definition 6.1, 6.2]{colding2003min}) in every annulus $An \in \mathcal{AN}_{r(x)}(x)$ for $x \in B_R \backslash l_z$. 

Let $x_0 \in B_R \backslash l_z$ and $An \in \mathcal{AN}_{r(x_0)}(x_0)$, we consider the following cases.
\begin{enumerate}[{Case} 1{:}]
\item \label{Case 1}
    $x_0 \in B_R \backslash S$: by shrinking $r(x_0)$ if necessary, we may assume $An \cap S = \emptyset$ for every annulus $An \in \mathcal{AN}_{r(x_0)}(x_0)$. Fix such small $An $ and define $\Grp(An) := \bigcup\limits_{h \in \Grp} h An$, which is the union of $2(g + 1)$ isometric copies of $An$ obtained by $\Grp$-symmetry. Let $\{ \Delta^{j, k}\}^k$ be a $\Grp$-equivariant minimizing sequence for Problem $(\Delta^j, \mathbf{Js}_{\Grp}(\Grp(An)))$ and converges to a varifold $V^j$. Observe that for any isotopy $\varphi \in \mathbf{Js}_{\Grp}(\Grp(An))$, $\varphi$ is uniquely determined by $\varphi |_{An}$, which is an isotopy in $\mathbf{Js}(An, An)$. Hence, Problem $(\Delta^j, \mathbf{Js}_{\Grp}(\Grp(An)))$ is equivalent to Problem $(\Delta^j \cap An, \mathbf{Js}(An, An))$, i.e., $\{ \Delta^{j, k} \cap An\}_k$ is also an minimizing sequence for Problem $(\Delta^j \cap An, \mathbf{Js}(An, An))$. By \cite[Theorem 7.3]{colding2003min}, there exists a stable minimal surface $\Gamma^j$ with $\bar{\Gamma}^j \backslash \Gamma^j \subset \partial An$ and $V^j = \Gamma^j$ in $An$. Since $V^j$ is $\Grp$-equivariant, $V^j = \Grp(\Gamma^j) : = \bigcup\limits_{h \in \Grp} h \Gamma^j$ in $\Grp(An)$. Also note that for any $j \in \mathbb{N}$:
    \begin{equation} \label{Replacement}
        \mathcal{H}^2_w(\Delta^j) \geq \inf_{\varphi \in \mathbf{Js}_{\Grp}(\Grp (An))}
     \mathcal{H}^2_w(\varphi(1, \Delta^{j})) = ||V^j|| \geq \inf_{\varphi \in \mathbf{Js}_{\Grp (B_R)}}\mathcal{H}^2_w(\varphi(1, \Sigma)).
    \end{equation}
After passing to a subsequence, we may assume that $V^j$ converges to a varifold $V^{*}$ and $V^*$ is a stable minimal surface in $G(An)$ by the compactness theorem for stable minimal surfaces \cite{choi1985space}. Moreover, taking $j \to \infty$ in  \eqref{Replacement}, we have
    \[
    ||V^*|| = \inf_{\varphi \in \mathbf{Js}_{\Grp}(B_R)} \mathcal{H}^2_w(\varphi(1, \Sigma))  = || V ||.
    \]
    Therefore, $V^*$ is stationary and is a replacement for $V$ in $\Grp(An)$ \cite[Definition 6.1]{colding2003min}. Now by a diagonal process, we can pick a subsequence $\{\Delta^{j, k(j)}\}_j$ of $\{ \Delta^{j, k}\}$ converging to $V^*$, repeating the same arguments above, we can find a further replacement $V^{**}$ for $V^*$ in $An$. And we can produce arbitrarily many further replacements for $V$ by the same process above. Hence, $V$ has the good replacement property in $An$. 
    
\item \label{Case 2}
    $x_0 \in S_1$, without loss of generality, we may assume $x_0 \in l_{\frac{\pi}{2(g+1)}}$. By shrinking $r(x) > 0$ if necessary, we may assume $An \cap (S \backslash l_{\frac{\pi}{2(g+1)}}) = \emptyset$. Proceeding the same as Case \ref{Case 1}, let $\{\Delta^{j, k}\}_k$ be a $\Grp$-equivariant minimizing sequence for Problem $(\Delta^j, \mathbf{Js}_{\Grp}(\Grp(An)))$, converging to a varifold $V^j$, where $\Grp(An) : = \bigcup\limits_{h \in \Grp} h An$. For any isotopy $\varphi \in \mathbf{Js}_{\Grp}(\Grp(An))$, $\varphi$ is uniquely determined by $\varphi |_{An}$, which is a $\mathbb{Z}_2$-isotopy in $\mathbf{Js}(An)$ where $\mathbb{Z}_2$ is the subgroup of $\Grp$ generated by the rotation of angel $\pi$ around the horizontal line $l_{\frac{\pi}{2(g+1)}}$. Then  $\{\Delta^{j, k} \cap An\}_k$ is a $\mathbb{Z}_2$-equivariant minimizing sequence for Problem $(\Delta^j \cap An, \mathbf{Js}_{\mathbb{Z}_2}(An, An))$ converging to $V^j \mres An$. Now observe that for any small ball $B \subset An$, if $B \cap l_{\frac{\pi}{2(g+1)}} = \emptyset$, $V^j\mres B$  is the varifold limit of the minimizing sequence  $\{\Delta^{j, k} \cap B\}_k$ for Problem ($\Delta_j \cap B$, $\mathbf{Js}(B, B)$) and $V^j$ is smooth in $B$  by Case \ref{Case 1}; if $B$ centered along $l_{\frac{\pi}{2(g+1)}}$, $V^j\mres B$  is the varifold limit of the $\mathbb{Z}_2$-equivariant minimizing sequence  $\{\Delta^{j, k} \cap B\}_k$ of Problem ($\Delta_j \cap B$, $\mathbf{Js}_{\mathbb{Z}_2}(B, B)$), since $g(An \cap \Delta_j) = 0$ for all $j \in \mathbb{N}$ sufficiently large, we can apply \cite[Proposition 7.3]{ketover2016free} to show that $V^j$ is smooth in $B$. Thus, $V^j$ is an embedded $\mathbb{Z}_2$-stable minimal surface in $An$. The only obstruction we have to apply the same arguments as Case \ref{Case 1} is that, we only have $\mathbb{Z}_2$-stability in $An$ for $V^j$ instead of the full stability. However, the proof of \cite[Proposition 6.3]{colding2003min} and arguments in Case \ref{Case 1}only make use of the following compactness theorem for stable minimal surfaces \cite[Section 2.5]{colding2003min}: a sequence of stable minimal surfaces has a convergent subsequence. We can replace this compactness theorem of \cite{schoen1983estimates} by the fact that any sequence of $\mathbb{Z}_2$-stable minimal surfaces with bounded area and genus  has a convergence subsequence converging smoothly away from finitely many points in the singular locus $l_1$ of $\mathbb{Z}_2$. And this fact is shown in \cite[Section 7.2]{ketover2016free}. Thus, the same arguments in Case \ref{Case 1} will conclude the good replacements property of $V$ in $An$.
    
    
\item \label{Case 3}
$x_0 \in S_2$. We may assume $x_0 \in P_{0}$ and $r(x_0) < dist(x_0, S \backslash P_0)$. Let $\{\Delta^{j, k}\}_k$ be a $\Grp$-equivariant minimizing sequence for Problem $(\Delta^j, \mathbf{Js}_{\Grp}(\Grp(An)))$, converging to a varifold $V^j$. Proceeding the same as Case \ref{Case 2},  to show the regularity of $V_j$ in $An$, it suffices to show $V_j$ is smoothly embedded in any small ball $B \subset An$ centered on $P_0$. Notice that $\{\Delta^{j, k} \cap B\}_k$ is a $\mathbb{Z}_2$-equivariant minimizing sequence for Problem $(\Delta^j \cap B,\mathbf{Js}_{\mathbb{Z}_2}(B,B))$ converging to $V^j \mres B$, where $\mathbb{Z}_2$ is the subgroup of $\Grp$ generated by the reflection around the vertical plane $P_0$. Let $B_+ : = B \cap \{(x, y, z)|x \geq 0\}$ be the part of $B$ lies on the component of $\mathbb{R}^3 \backslash P_0$ with positive $x$-coordinates. It is easy to see that any $\mathbb{Z}_2$-equivariant vector field $\chi$ on $B$ is uniquely determined by  $\chi |_{B_+}$ and satisfies $\chi(x) \cdot v(x) = 0$ for any $x \in P_0 \cap B$ where $v$ is the unit outer normal of $B_+$ on   $P_0 \cap B$. Thus, Problem $(\Delta^j \cap B,\mathbf{Js}_{\mathbb{Z}_2}(B,B))$ is equivalent to a minimization problem on $B_+$ with partially free boundary (see \cite[Definition 7.2]{Li2015General}). By \cite[Theorem 7.3]{Li2015General}, $V^j$ is smoothly embedded in $B$ and orthogonal to $P_0$. Thus, $V^j$ is smoothly embedded in $An$. Now replacing the definition of replacements and good replacements property \cite[Definition 6.1, 6.2]{colding2003min} by \cite[Definition 8.1, 8.2]{Li2015General} and proceeding the same as in Case \ref{Case 1}, we conclude that  $V$ has good replacements property in $An$ and $V$ is embedded on $P_0 \backslash \partial B_R $ by \cite[Theorem 8.3]{Li2015General}.
\end{enumerate}

On the other hand, let $x_0 \in (\text{supp}||V|| \cap \partial B_R) \backslash S$ and take a small ball $B_x$ away from $S$. We can apply \cite[Lemma 8.1]{de2010genus} to conclude that $V$ has smooth boundary in $B_x$ since the proof there is purely a local argument. Now it remains to show the regularity of $V$ on $\mathcal{W}:= (\text{supp}||V|| \cap S) \cap (l_z \cup \partial B_R)$. We claim that $\mathcal{W}$ is a finite set.

Indeed, due to the symmetry and that $\Theta:= \text{supp}||V|| \mres (B_R \backslash l_z)$ is an embedded minimal surface, $\Theta$ must be orthogonal to the vertical planes $P_{\frac{j \pi}{g+1}}$ for all $j \in \{0, \dots, g\}$. Consequently, by Sard's theorem, $\Theta \cap P_{\frac{j \pi}{g+1}}$ is a finite collection of disjoint simple curves $\{\gamma_1, \dots, \gamma_{N_j} \}$. Then $\mathcal{X}_j:= \text{supp}||V|| \cap P_{\frac{j \pi}{g+1}} \cap (\partial B_R \cup l_z) \subset \bigcup_{k =1}^{N_j}\partial\gamma_{j}$ is a finite set of points for each $j$. Also observe that $\mathcal{Y}:= \partial B_R \cap (\bigcup_{j = 1}^{g + 1} l_{\frac{(2j-1)\pi}{2(g+1)}})$ is a finite set. Thus, $\mathcal{W} \subset \bigcup_{j =0}^{g}\mathcal{X}_j \cup \mathcal{Y}$ has finite cardinality.

By the removable singularity theorem in \cite[Proposition 1]{choi1985space}, we can extend smoothly across $\mathcal{W}$ and have the full regularity of $V$. Thus, $V = \bigcup_{k = 1}^N n_k\Gamma_k$ is a union of  connected disjoint embedded minimal surfaces $\Gamma_1, \dots, \Gamma_N \subset B_R$ with respect to the metric $g_{Exp}$ with integer multiplicity $n_k $. Moreover, we know  that either $\partial \Gamma_k = 0$ or $\partial \Gamma_k$ is the union of some connected components of $\partial \Sigma$. In the latter case, the multiplicity of $\Gamma_k$ is necessarily 1. On the other hand, $\partial \Gamma_k = 0$ cannot occur as there are no closed self-expanders. The genus bound then follows easily from \cite[Proposition 4.7]{Ketover2019Genus}. 
    
\end{proof}

\input{output.bbl}

\end{document}